\documentclass[a4paper,12pt,reqno]{amsart}
\usepackage{anysize} \marginsize{0.7in}{0.7in}{1in}{1in}

\makeatletter
\DeclareMathAlphabet{\mathpzc}{OT1}{pzc}{m}{it}

\usepackage{placeins}% to avoid figures split theorems

\usepackage{amsmath, amssymb, amsfonts, amstext, verbatim, amsthm, mathrsfs}
\usepackage[mathcal]{eucal}
\usepackage{microtype}
\usepackage{mathdots}
\usepackage{caption,subcaption}
\renewcommand{\geq}{\geqslant}
\usepackage[usenames]{color}
\usepackage{aliascnt} 
\usepackage{epsfig}
\usepackage[all,arc,knot,poly]{xy}
\usepackage{graphicx}
\usepackage[usenames]{color}

\usepackage[colorlinks=true,linkcolor=blue,citecolor=blue,urlcolor=blue,citebordercolor={0 0 1},urlbordercolor={0 0 1},linkbordercolor={0 0 1}]{hyperref} 
\usepackage{enumerate}
\usepackage{xspace}
\usepackage{comment}

\usepackage{tikz}
\usepackage{pgfplots}
\usetikzlibrary{decorations.pathreplacing,cd,bending}
\usetikzlibrary{matrix,arrows}

\theoremstyle{plain}

\newcommand{\refnewtheoremn}[4]{
\newaliascnt{#1}{#2}
\newtheorem{#1}[#1]{#3}
\aliascntresetthe{#1}
\expandafter\providecommand\csname #1autorefname\endcsname{#4}}

\newcommand{\refnewtheorem}[3]{\refnewtheoremn{#1}{#2}{#3}{#3}}

\def\makeCal#1{
\expandafter\newcommand\csname c#1\endcsname{\mathcal{#1}}}
\def\makeBB#1{
\expandafter\newcommand\csname b#1\endcsname{\mathbb{#1}}}
\def\makeFrak#1{
\expandafter\newcommand\csname f#1\endcsname{\mathfrak{#1}}}

\count@=0
\loop
\advance\count@ 1
\edef\y{\@Alph\count@}
\expandafter\makeCal\y
\expandafter\makeBB\y
\expandafter\makeFrak\y
\ifnum\count@<26
\repeat
\newtheorem{thm}{Theorem}[section]

\refnewtheorem{lemma}{thm}{Lemma}
\refnewtheorem{assumptions}{thm}{Assumptions}
\refnewtheorem{examples}{thm}{Examples}
\refnewtheorem{claim}{thm}{Claim}
\refnewtheorem{cor}{thm}{Corollary}
\refnewtheorem{conj}{thm}{Conjecture}
\refnewtheorem{prop}{thm}{Proposition}
\refnewtheorem{proposition}{thm}{Proposition}
\refnewtheorem{question}{thm}{Question}
\refnewtheorem{assumption}{thm}{Assumption}

\theoremstyle{definition}
\refnewtheorem{remark}{thm}{Remark}
\refnewtheorem{remarks}{thm}{Remarks}
\refnewtheorem{properties}{thm}{Properties}
\refnewtheorem{defn}{thm}{Definition}
\refnewtheorem{definition}{thm}{Definition}
\refnewtheorem{notn}{thm}{Notation}
\refnewtheorem{const}{thm}{Construction}
\refnewtheorem{example}{thm}{Example}
\refnewtheorem{fact}{thm}{Fact}
\refnewtheorem{problem}{thm}{Problem}

\newcommand{\Coh}{\operatorname{Coh}}
\newcommand{\Stab}{\operatorname{Stab}}
\newcommand{\Hom}{\operatorname{Hom}}
\newcommand{\isom}{\cong}
\renewcommand{\O}{\mathscr{O}}

\newcommand{\SL}{\operatorname{SL}}
\newcommand{\GL}{\operatorname{GL}}
\newcommand{\D}{D}
\newcommand{\Forg}{\operatorname{Forg}}
\newcommand{\rep}{\operatorname{rep}}

\newcommand{\HHilb}{\operatorname{Hilb}^H}

\newcommand{\ch}{\operatorname{ch}}
\renewcommand{\dim}{\operatorname{dim}}
\newcommand{\tensor}{\otimes}

\newcommand{\DT}{\operatorname{DT}}

\title[Invariant stability conditions]{Invariant stability conditions on certain local Calabi-Yau threefolds}

\author{Tom Bridgeland}
\address[T. Bridgeland]{Department of Pure Mathematics\\ University of Sheffield\\ Hicks Building, Hounsfield Road\\ Sheffield, S3 7RH\\ UK}
\email{t.bridgeland@sheffield.ac.uk}

\author{Fabrizio Del Monte}
\address[Fabrizio Del Monte]{School of Mathematics \\ University of Birmingham \\ Watson Building \\ Edgbaston \\  Birmingham B15 2TT \\ UK}
\email{f.delmonte@bham.ac.uk}

\author{Luca Giovenzana}
\address[Luca Giovenzana]{Department of Pure Mathematics\\ University of Sheffield\\ Hicks Building, Hounsfield Road\\ Sheffield, S3 7RH\\ UK}
\email{l.giovenzana@sheffield.ac.uk}

\date{}

\begin{document}

\begin{abstract}
We apply results on inducing stability conditions to local Calabi-Yau threefolds and obtain applications to Donaldson-Thomas (DT) theory. 
 A basic example is the total space of the canonical bundle of $Z=\bP^1\times \bP^1$. We  use a result of Dell to construct stability conditions on the derived category of $X$ for which all stable objects can be explicitly described. We relate them to stability conditions on the resolved conifold $Y=\O_{\bP^1}(-1)^{\oplus 2}$ in two ways: geometrically via the McKay correspondence, and algebraically via a  quotienting operation on quivers with potential.  These stability conditions were first discussed in the physics literature  by Closset and del Zotto, and were constructed mathematically  by  Xiong by a different method. We obtain a complete description of the corresponding DT invariants, from which we can conclude that they define analytic wall-crossing structures in the sense of Kontsevich and Soibelman.    In the last section we discuss several other examples of a similar flavour. 
\end{abstract}

\maketitle

\section{Introduction}

Let $Z$ be a smooth rational surface  and denote by $X=\omega_Z$ the total space of its canonical bundle. 
Such quasi-projective Calabi-Yau threefolds  have been much-studied by both mathematicians and physicists. They are more amenable than compact examples such as quintic threefolds $X\subset \bP^4$ but richer than models such as the resolved conifold $X=\O_{\bP^1}(-1)^{\oplus 2}$ which contain no compact divisors.
This intermediate level of complexity is reflected in our understanding of the (bounded, compactly supported)  derived category of  coherent sheaves $\D_c(X)$ and the associated space of stability conditions $\Stab(X)$ \cite{Stab}.

For the resolved conifold  we have a complete description of a connected component of $\Stab(X)$, and for each stability condition in this component the subcategory  of stable objects can be explicitly described \cite{Toda}. In contrast,  for a quintic  threefold the existence of stability conditions has only been demonstrated relatively  recently \cite{Li19}, and almost nothing is known about the global geometry of $\Stab(X)$. The case of local threefolds  $X=\omega_Z$ lies somewhere between these extremes. It is easy to construct examples of stability conditions,  and  in special cases we even know something about the global geometry of $\Stab(X)$ \cite{BM,BlP}.
Nonetheless,  a complete description of the space of stability conditions is lacking, and for a generic stability condition the category of stable objects is extremely complicated, giving rise for example to moduli spaces of arbitrarily high dimensions. 

It came as a surprise therefore, when work by Longhi and one of us \cite{DML2023}, building upon earlier
observations of  Closset and del Zotto \cite{CZ}, suggested that for some rational surfaces $Z$ there is a special locus in the  space of  stability conditions on $X=\omega_Z$ whose geometry can be completely understood, and  for which it is possible to explicitly describe all stable objects. 
The simplest example occurs when  $Z=\bP^1\times \bP^1$. In this case the special stability conditions have a spectrum of stable objects  closely resembling that for the resolved conifold.   These stability conditions were later constructed mathematically by Yirui Xiong \cite{Xiong2024} using an argument based on that of \cite{DML2023}, but the essential mystery remained: why do there exist such simple stability conditions on $X$?

The aim of this note is to give a simple answer to this question using an existing result on induced  stability conditions \cite{Dell, MMS,Pol}. We show that a finite group acts  on the stability space of $X$, and the special stability conditions referred to above are precisely those that are invariant under this action. In fact we can construct the group action in two ways: geometrically using the McKay correspondence, or  algebraically via  symmetries of quivers with potential. After explaining the two approaches  we will describe several examples of the basic phenomenon and discuss some applications to  Donaldson-Thomas (DT) invariants. 

\subsection{Inducing stability conditions}
We now give  a brief summary of our two main constructions, focusing on the the basic example  $Z=\bP^1\times \bP^1$. We refer the reader to the  body of the paper for more precise definitions and statements of results.

The  geometric approach is as follows. The local threefold $X=\omega_Z$ arises as the crepant resolution of the quotient of the resolved conifold $Y=\O_{\bP^1}(-1)^{\oplus 2}$ by the finite group $H=\bZ_2$ acting by multiplication by $-1$ on the fibres of the projection to $\bP^1$. The derived McKay correspondence \cite{BKR} then gives an equivalence of (bounded, compactly-supported) derived categories 
$\Phi\colon \D_c(X)\to \D^H_c(Y)$.
The equivariant derived category $\D^H_c(Y)$  carries an action of the  group of characters $G=\Hom_{\bZ}(H,\bC^*)$ by tensor product. Using $\Phi$ we can transfer this to an action of $G$ on $\D_c(X)$. A result of Dell \cite{Dell}, which in this geometric setting was previously obtained by Polishchuk \cite{Pol},  then gives an isomorphism of fixed loci
\[\Stab(X)^G\isom \Stab(Y)^{H}.\]
There is a standard  connected component $\Stab_0(Y)\subset \Stab(Y)$ which is well-understood, and it is easy to see that the action of $H$ on this component is in fact trivial. Thus we obtain a connected component $\Stab_0(X)^G$ of the fixed locus $\Stab(X)^G$ and 
 an identification \begin{equation}\label{sheaf}\Stab_0(X)^G\isom \Stab_0(Y).\end{equation}
The special stability conditions referred to  above are those contained in the closed subset $\Stab_0(X)^G\subset \Stab(X)$, and will henceforth be referred to as \emph{invariant stability conditions}. In \cite{DML2023B} the subset $\Stab_0(X)^G$ was called the fine-tuned stratum of $\Stab(X)$.

We can approach the same stability conditions algebraically  using quivers with potential. The exceptional collection $\big[\O,\O(1,0),\O(1,1),\O(2,1)\big]$
on $Z$ pulls back  to give a tilting bundle on $X=\omega_Z$. This leads to a derived equivalence between  coherent sheaves on $X$ and  representations of the Jacobi algebra of the quiver with potential $(Q,W)$ whose underlying quiver $Q$ is shown in Figure \ref{fig_one}(a). Similarly, pulling back the exceptional collection $\big[\O,\O(1)\big]$ on $\bP^1$  gives a tilting bundle on $Y=\O_{\bP^1}(-1)^{\oplus 2}$ whose endomorphism  algebra is described by a quiver with potential $(Q',W')$ whose quiver $Q'$ is shown in Figure \ref{fig_one}(b).\medskip

\begin{figure}[h!]
\centering
\begin{subfigure}[m]{0.3\linewidth}
    \centering
    \includegraphics[width=\linewidth]{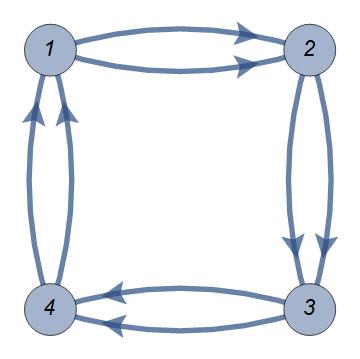}
   \caption{$Q$}
 \end{subfigure}\qquad\qquad\qquad
 \begin{subfigure}[m]{0.3\linewidth}
    \centering
    \includegraphics[width=\linewidth]{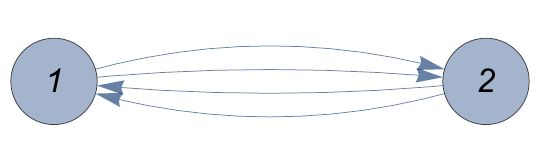}
\caption{$Q'$}
 \end{subfigure}
 \caption{Quivers for $X=\omega_{\bP^1\times\bP^1}$ and $Y=\O_{\bP^1}(-1)^{\oplus 2}$}\label{fig_one}
\end{figure}
The potentials are respectively
\begin{equation}
    W=-X_{1,2}^{(1)} X_{2,3}^{(1)} X_{3,4}^{(1)}X_{4,1}^{(1)}-X_{1,2}^{(2)} X_{2,3}^{(2)} X_{3,4}^{(2)} X_{4,1}^{(2)}+X_{1,2}^{(2)} X_{2,3}^{(1)} X_{3,4}^{(2)} X_{4,1}^{(1)}+X_{1,2}^{(1)} X_{2,3}^{(2)} X_{3,4}^{(1)} X_{4,1}^{(2)},
\end{equation}
and
\begin{equation}
    W'=-X_{1,2}^{(1)} X_{2,1}^{(1)} X_{1,2}^{(2)}X_{2,1}^{(2)}+X_{1,2}^{(2)} X_{2,1}^{(1)} X_{1,2}^{(1)} X_{2,1}^{(2)},
\end{equation}
where $X_{i,j}^{(k)}$ denotes the  $k$-th arrow from the vertex $i$ to the vertex $j$.

Rotation by a half turn defines a free action of the  group $G=\bZ_2$  on the quiver $Q$ inducing the permutation $(1,3)(2,4)$ on the nodes. This action preserves  the potential  $W$, and the pair  $(Q',W')$ arises naturally as the quotient. It follows  that the category of representations of $(Q,W)$ is  equivalent to the category of $H$-equivariant representations  of  $(Q',W')$,  where the group of characters $H=\Hom_{\bZ}(G,\bC^*)$  acts on the Jacobi algebra of $(Q',W')$ via  certain rescalings of the arrows.

We consider the space of stability conditions $\Stab(Q,W)$ on the  (bounded, finite-dimensional) derived category of the Jacobi algebra of the quiver with potential $(Q,W)$, and similarly for $(Q',W')$.
 Dell's result  gives an isomorphism of fixed loci
\[\Stab(Q,W)^G\isom \Stab(Q',W')^H,\] and once again the action of $H$ on the standard connected component $\Stab_0(Q',W')\subset \Stab(Q',W')$ is trivial. Thus  we obtain a connected component  $\Stab_0(Q,W)^G\subset \Stab(Q,W)^G$ and an identification
 \[\Stab_0(Q,W)^G\isom \Stab_0(Q',W').\]
It is not hard to show that this identification corresponds to \eqref{sheaf}  under the tilting equivalence relating  coherent sheaves to quiver representations. 

In Section \ref{sec:ex} we consider several other examples of the same phenomenon:
\begin{itemize}
    \item [(i)] A pseudo del Pezzo surface $Z=PdP_5$ obtained by blowing up $\bP^2$ in a certain non-generic set of 5 points. The   local threefold $X=\omega_Z$ is a crepant resolution of a quotient of the resolved conifold $Y=\O_{\bP^1}(-1)^{\oplus 2}$ by the group  $G=\bZ_2\times\bZ_2$.\smallskip
    
    \item[(ii)] The del Pezzo surface $Z=dP_3$ obtained by blowing up $\bP^2$ in 3 distinct points. We can define two distinct invariant loci in the stability space of the local threefold $X=\omega_Z$ by using  actions of the groups $G=\bZ_2$ or $G=\bZ_3$ on an appropriate   quiver with potential.\smallskip
    
    \item[(iii)] An infinite family of toric $CY_3$ geometries $X$  usually referred to as $Y^{N,0}$. The space $Y^{N,0}$ is a crepant resolution of a $\bZ_N$ quotient of the resolved conifold $Y=\O_{\bP^1}(-1)^{\oplus 2}$, the case $N=2$ coinciding with our basic example $X=\omega_{\bP^1\times\bP^1}$. 
\end{itemize}

\subsection{Donaldson-Thomas invariants}

 The result we use to induce stability conditions provides information not only about the spaces of stability conditions but also about the categories of semistable objects. In the case  $Z=\bP^1\times \bP^1$ described above, the central charges of the semistable objects lie on  the collection of rays shown in Figure~\ref{fig:raydiag}.

The categories of semistable objects on all rays but the central one are easily understood: there are just two  stable objects which are moreover spherical and have no extensions between them. The category of semistable objects on the central ray is equivalent to the category of zero-dimensional $H$-equivariant sheaves on $Y$, and the relevant DT invariants can be read off  from the orbifold vertex calculations of Bryan, Cadman and Young \cite{BCY}. We obtain

 \begin{thm} (= Theorem \ref{thm:DTP1P1}.)
 \label{thm_intro}
 Let $X=\omega_Z$ be the local threefold  corresponding to the del Pezzo surface  $Z=\bP^1\times \bP^1$, and let $(Q,W)$ be the quiver with potential  of Figure \ref{fig_one}(a). Let $\sigma\in \Stab(Q,W)^G$  be an invariant stability condition as above.  Then 
    the nonzero DT invariants are
    \[\Omega_{\sigma}((n+1)\gamma_1+n\gamma_2)= \Omega_{\sigma}(n\gamma_1+(n+1)\gamma_2)=1, \qquad n\in \bZ,\]\[\Omega_{\sigma}((n+1)\gamma_3+n\gamma_4)=\Omega_{\sigma}(n\gamma_3+(n+1)\gamma_4)=1,\qquad n\in \bZ,\]
    \[\Omega_{\sigma}((n+1)(\gamma_1+\gamma_2)+n(\gamma_3+\gamma_4))=\Omega_{\sigma} (n(\gamma_1+\gamma_2)+(n+1)(\gamma_3+\gamma_4))=-2, \qquad n\in \bZ,\]\[ \Omega_{\sigma}( n(\gamma_1+\gamma_2+\gamma_3+\gamma_4))=-4, \qquad n\in \bZ\setminus\{0\},\]
    where $\gamma_1,\gamma_2,\gamma_3, \gamma_4\in \bZ^{Q_0}$ are the classes defined by the vertices of the quiver.
\end{thm}

The wall-crossing formula shows that in principle Theorem \ref{thm_intro} determines the DT invariants for all  stability conditions in the same connected component. Note that for a generic stability condition these could be very complicated, containing information about virtual Euler characteristics of moduli spaces of semistable sheaves on $Z$ of arbitrary Chern character.

 \begin{figure}[h]
     \centering
\begin{tikzpicture}[scale=1.5]

\foreach \k in {-10,-9,...,10} {
    \draw[thick, ->] (0,0) -- (\k/2+0.2,2);
    \draw[thick, ->] (0,0) -- (\k/2-0.2,-2);
}

% Red arrows
\foreach \n in {-2,-1,1} {
    \draw[red, thick, ->] (0,0) -- (1+2*\n,0);
\draw[red, thick, ->] (0,0) -- (-1-2*\n,0);
}

% Red arrows
\foreach \m in {-2,-1,1,2} {
    \draw[red, thick, ->] (0,0) -- (2*\m,0);
}

% Labels
\node at (3,2.5) {$Z_{\gamma_1+k(\gamma_1+\gamma_2)} = Z_{\gamma_3+k(\gamma_3+\gamma_4)}$};
\node at (-3,-2.5) {$Z_{\gamma_2+k(\gamma_1+\gamma_2)} = Z_{\gamma_4+k(\gamma_3+\gamma_4)}$};
\node at (4,0.5) {\color{red}$Z_{\gamma_1+\gamma_2+k\delta} = Z_{\gamma_3+\gamma_4+k\delta}$};
\node at (3,-0.5) {\color{red}$Z_{k\delta}$};

\end{tikzpicture}
     \caption{Ray diagram for the invariant stability condition on $\omega_{\mathbb{P}^1\times\mathbb{P}^1}$.}
     \label{fig:raydiag}
 \end{figure}
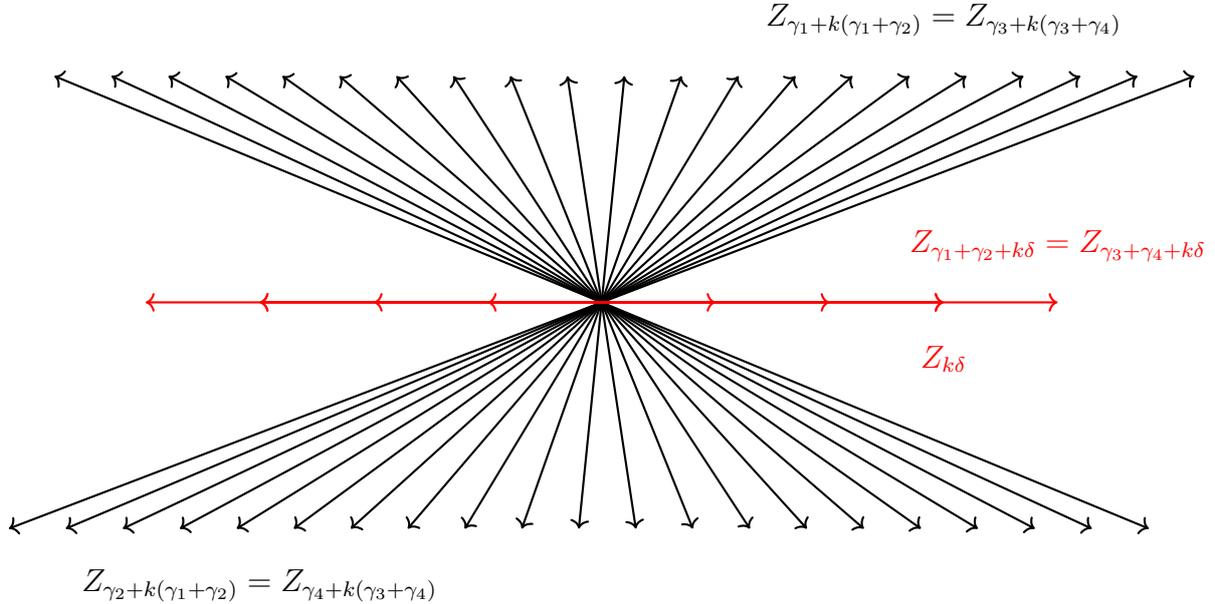

\subsection{Further applications}
One interesting application of Theorem \ref{thm_intro} uses a deep result of Kontsevich and Soibelman, who showed \cite[Section 3.5]{KS} that a certain property relating to growth rates of DT invariants depends only on the connected component containing a given stability condition. Since it is immediate from Theorem \ref{thm_intro} that this  property  holds for invariant stability conditions, we have the following result:

\begin{thm}
    For the local threefold $X=\omega_Z$ corresponding to the del Pezzo surface  $Z=\bP^1\times \bP^1$ there is a connected component of $\Stab(X)$ such that the DT theory associated to any stability condition in this component defines an analytic wall-crossing structure.
\end{thm}
These examples are the first known instances of analytic wall-crossing structures associated with  local Calabi-Yau threefolds containing compact divisors. Further examples are given by the $Y_{N,0}$ geometries which we treat in Section \ref{banana}.

Another application relates to the DT Riemann-Hilbert Problem (RHP) discussed in \cite{Bridgeland2019}.  First, observe that invariant stability conditions determine a convergent, integral BPS structure in the sense of \emph{loc. cit.}, so that the DT RHP discussed there is well-defined. It was explained in \cite{DelMonte2024} that the corresponding DT RHP admits the following trivial solution:
\[
    X(\gamma) = \exp(Z(\gamma)/\epsilon),
\]
where $Z$ is the central charge function of the invariant stability condition, and $\epsilon$ is an additional parameter introduced in the definition of the RHP.

This statement follows easily (see the proof of \cite[Lemma 4.8]{Bridgeland2019}) from the following claim:  for an invariant stability condition, the partially-defined  automorphisms $\mathbb{S}_\sigma(\ell)$ of the twisted torus $\mathbb{T}$ which describe the jumps in the RHP are all trivial. To see this, recall that $\bS_\sigma(\ell)$ acts on a twisted character  $X(\beta)$ by the formula
\begin{equation}
\label{hi}
    \mathbb{S}_\sigma(\ell)^*(X(\beta)) = X(\beta) \cdot \prod_{Z(\gamma) \in \ell} (1 - X(\gamma))^{\Omega_\sigma(\gamma) \langle \beta, \gamma \rangle},
\end{equation}
where $\Omega_\sigma(\gamma)$ are the BPS invariants, and $\langle -, - \rangle\colon \Gamma\times \Gamma\to \bZ$ is the skew-symmetric Euler pairing on the charge lattice $\Gamma=\bZ^{Q_0}$.  By definition, an invariant stability condition $\sigma=(Z,\cP)\in \Stab(X)$  is invariant under the action of $G$ on the category $\D_c(X)$. The induced action of $G$ on the lattice $\Gamma$  therefore preserves the  central charge $Z\colon \Gamma\to \bC$ and the DT invariants $\Omega_{\sigma}(\gamma)$. Crucially, in this example we also have $\gamma+g(\gamma)\in \ker \langle-,-\rangle$ for any $\gamma\in \Gamma$. The claim $\bS_\sigma(\ell)=\operatorname{id}$ then  follows by combining the terms for $\gamma$ and $g(\gamma)$ in the product  \eqref{hi}. 

Although the solution to the RHP for the invariant stability conditions is trivial, it was shown in \cite{DelMonte2024,DML2023B} that in the vicinity of these points, there exists a codimension-one complex submanifold where the DT invariants remain unchanged, but the solution becomes nontrivial. According to the approach in \cite{DelMonte2024}, the trivial solutions to the RHP correspond to algebraic solutions of cluster integrable systems. In contrast, \cite{DML2023B} argues that more general solutions of the cluster integrable system are associated with the previously mentioned codimension-one locus. For example, in the case of del Pezzo surfaces, these more general solutions correspond to q-Painlevé equations \cite{Bershtein2018}, while for the $\mathbb{Z}_N$ orbifolds of the conifold discussed in Section \ref{banana}, they correspond to non-autonomous cluster Toda chains \cite{Bershtein2019}. Note also that, even at the invariant points, a non-trivial RHP can be obtained by considering non-invariant constant terms, or by doubling the lattice $\Gamma$, as  in \cite[Section 2.8]{Bridgeland2019}.

\section{Geometric approach}

In this section we explain how to use the derived McKay correspondence \cite{BKR} to construct stability conditions on certain local threefolds. We rely on techniques for inducing stability conditions developed by Macr{\`i}, Mehrotra and Stellari \cite{MMS}, and by Polishchuk \cite{Pol}. The exact formulation we use is due to Dell \cite{Dell}. 

\subsection{Conventions and notation}
\label{defs}

All varieties will be over the complex numbers. Suppose $M$ is  a smooth quasi-projective variety. We define $\Coh(M)$ to be the abelian category of coherent sheaves on $M$ and $\Coh_c(M)\subset \Coh(M)$ to be the full subcategory of sheaves with compact support. We define $\D_c(M)$ to be the full triangulated subcategory of the bounded derived category $\D(M)=\D^b(\Coh(M))$ consisting of complexes  whose cohomology sheaves  lie in the subcategory $\Coh_c(M)$.

 Suppose a finite  group $G$ acts on $M$.
 We denote by $\Coh^G(M)$ the abelian category of $G$-equivariant sheaves.  We say that an equivariant sheaf  is compactly supported if the underlying sheaf is. We define $\D^G_c (M)$ to be the full triangulated subcategory of the bounded derived category $\D^b(\Coh^G(M))$ consisting of complexes whose cohomology sheaves have compact support. There is an obvious exact forgetful functor $\Forg\colon \Coh^G(M)\to \Coh(M)$ which induces a  triangulated  functor $\Forg\colon \D^G_c(M)\to \D_c(M)$.

We assume from now on that $G$ is abelian and write $\Hat{G}=\Hom_{\bZ}(G,\bC^*)$ for the group of characters. Each element $\chi\in \Hat{G}$ defines a rank 1 equivariant sheaf $\O_M\tensor_{\bC} \chi$. Tensor product with this sheaf defines an exact functor $T(\chi)\colon \Coh^G(M)\to \Coh^G(M)$. This gives rise to an action of $\Hat{G}$ by  exact auto-equivalences on the category $\Coh^G(M)$ preserving the subcategory of sheaves with compact support. Taking derived functors we obtain a corresponding action of  $\Hat{G}$ by triangulated auto-equivalences on  the category $\D^G_c(M)$.

A stability condition $\sigma=(Z,\cP)$ on a triangulated category $\D$ consists of two pieces of data:
\begin{itemize}
    \item[(i)] for each $\phi\in \bR$ a full subcategory $\cP(\phi)\subset \D$,
    \item[(ii)] a group homomorphism $Z\colon K_0(\D)\to \bC,$
    \end{itemize}
    satisfying a collection of axioms \cite{Stab}. The objects of $\cP(\phi)$ are said to be semistable of phase $\phi$, and the map $Z$ is called the central charge.
    We denote by $\Stab(\D)$ the space of locally-finite stability conditions  on a triangulated category $\D$. Given a finite group $G$ acting by triangulated autoequivalences on $\D$ we denote by $\Stab(D)^G$ the fixed locus for the induced action of $G$ on $\Stab(\D)$.

Suppose that $M$ is a smooth quasi-projective variety. The expression
\[\chi(E,F)=\sum_{i\in \bZ}(-1)^i \dim_{\bC} \Hom(E,F[i])\] descends to a well-defined Euler pairing on Grothendieck groups
\[\chi(-,-)\colon K(\D(M))\times K_0(\D_c(M))\to \bZ.\]
We say that a stability condition  $\sigma=(Z,\cP)$ on the category $\D_c(M)$ is \emph{quasi-numerical} if the central charge $Z$ vanishes on the subgroup $\ker \chi(-,-)\subset K_0(\D_c(M))$.
If $M$ is projective this coincides with the usual notion of a numerical stability condition. 
We write  $\Stab(M)\subset \Stab(\D_c(M))$ for the subset of locally-finite quasi-numerical stability conditions.

\subsection{Inducing stability conditions}

Let $Y$ be a smooth quasi-projective threefold, and let the finite abelian group $H$ act faithfully on $Y$. We assume that the canonical bundle is locally trivial as an $H$-equivariant bundle. This implies  that at each fixed point $y\in Y$ the induced action of $H$ on the tangent space $T_y Y$ factors via the subgroup $\SL(T_{Y,y})\subset \GL(T_{Y,y})$. The quotient $Y/H$ is then a quasi-projective Gorenstein threefold \cite{Kh,Wa}.
Let $X=\HHilb(Y)$ be the Hilbert scheme parameterising $H$-clusters on $Y$, and let  $f\colon X\to Y/H$ denote the natural Hilbert-Chow map. The following is the main result of \cite{BKR}.

\begin{thm}
\label{mckay}
    The map $f$ is a crepant resolution of singularities. Moreover the Fourier-Mukai transform given by the universal $H$-cluster defines an equivalence of categories
    \begin{equation}\label{phi}\Phi\colon \D(X)\to \D^H(Y).\end{equation}
    which restricts to an equivalence on the subcategories of objects with compact support.
\end{thm}

Let $G=\Hom_{\bZ}(H,\bC^*)$ be the group of characters. This group acts by auto-equivalences on $\D^H(Y)$ as above and we can use $\Phi$ to transfer this action to $D(X)$. These actions preserve the subcategories of objects with compact support. We define
\begin{equation}\label{th}\Theta=\Forg\circ\, \Phi\colon \D(X)\to \D(Y).\end{equation}
 
Using this functor to induce stability conditions on $\D_c(X)$ gives the following result.

\begin{thm}
\label{main}
    In the above situation there is
an isomorphism of complex manifolds
\begin{equation}
    \label{f}
F\colon \Stab(Y)^{H}\to \Stab(X)^{G}.\end{equation}
 Moreover, given a stability condition $\sigma\in \Stab(Y)^H$ 
 
 \begin{itemize}
     \item [(i)]  an object $E\in D_c(X)$ is semistable of phase $\phi\in \bR$ for the stability condition $F(\sigma)$  precisely if $\Theta(E)\in D_c(Y)$ is semistable of phase $\phi$ for the stability condition $\sigma$,
     \item [(ii)] the central charge of an object $E\in \D_c(X)$ in the stability condition $F(\sigma)$ is equal to the central charge of the object $\Theta(E)\in \D_c(Y)$ in the stability condition $\sigma$. 
 \end{itemize}

\end{thm}

\begin{proof}
Applying the result of Dell \cite[Lemma 2.23]{Dell} with $G=H$ and $\cD=D_c(Y)$ shows that  the forgetful functor $\D^H_c(Y)\to \D_c(Y)$ induces a continuous bijection
\begin{equation}
    \label{hann}
F\colon \Stab(D_c(Y))^{H}\to \Stab(D^H_c(Y))^{G}.\end{equation}
Combining this with  the McKay equivalence \eqref{phi} immediately gives a continuous bijection
\begin{equation*}F\colon \Stab(D_c(Y))^{H}\to \Stab(D_c(X))^{G}\end{equation*}
satisfying conditions (i) and (ii) of the statement.

At the level of central charges this bijection  arises from an isomorphism  of vector spaces
\begin{equation}
    \label{done?}
K_0(\D_c(Y))^H\tensor_{\bZ}\bQ\isom K_0(\D_c(X))^G\tensor_{\bZ} \bQ.\end{equation} By  construction, this isomorphism  is induced by the functor \eqref{th}. Moreover, it follows from the results of Elagin \cite{Ela} that, up to a factor of $|G|$, the inverse map is induced by the adjoint of this functor. Since both these functors are of Fourier-Mukai type, it is then easy to see  that \eqref{hann} preserves the subsets of quasi-numerical stability conditions. Thus \eqref{hann} restricts to  a continuous bijection \eqref{f}.

The complex structure on the space of stability conditions is defined by pullback along the local homeomorphism to the space of central charges. Since \eqref{done?} is linear isomorphism  it follows that the bijection \eqref{f} is an isomorphism of complex manifolds.
\end{proof}

We give one example of this result now, more will be given below.

\begin{example}
\label{egg}
    Let $H=\bZ_3$ act linearly on $Y=\bC^3$ via multiplication by a non-trivial third root of unity. Then $X=\HHilb(Y)$ is a crepant resolution of $Y/H$ and is therefore isomorphic to the local del Pezzo threefold $X=\omega_Z$ associated to $Z=\bP^2$. There is a connected component $\Stab_0(Y)\subset \Stab(Y)$ which is naturally identified with $\bC$. In the stability condition corresponding to $w\in \bC$ each skyscraper sheaf is stable of phase $\phi=\operatorname{Re}(w)$ and has central charge $Z(\O_y)=e^{\pi i w}$, and these are the only stable objects up to shift. The induced action of $H$ on $\Stab_0(Y)$ is trivial so Theorem \ref{main} gives a connected component $\Stab_0(X)^G\subset \Stab(X)^G$ and an identification $\Stab_0(X)^G\isom \bC$.
\end{example}

\section{Algebraic approach}\label{sec:AlgApproach}

In this section we describe a different approach to constructing  stability conditions on local threefolds. We again use the results on inducing stability conditions of \cite{Dell, MMS, Pol} but this time apply them to categories of representations of quivers with potential.

\subsection{Symmetries of quivers}
By a quiver we mean a finite directed graph which we specify in the usual way by a set of vertices $Q_0$, a set of arrows $Q_1$, and source and target maps $s,t\colon Q_1\to Q_0$.  We denote the corresponding path algebra by $\bC Q$. By a quiver with relations we mean a quiver $Q$ together with a two-sided ideal  $I\subset \bC Q$  generated 
by paths of length $\geq 2$. By a representation of $(Q,I)$ we mean a finite-dimensional left  module of the quotient algebra $A=\bC Q/I$. We  write  $\rep(Q,I)$ for the category of such representations and  $\D(Q,I)=D^b(\rep(Q,I))$ for its bounded derived category.

There is a group homomorphism $\dim\colon K_0(\D(Q,I))\to \bZ^{Q_0}$ sending a representation to its dimension vector. We  denote by $\Stab(Q,I)\subset \Stab(\D(Q,I))$ the closed subspace consisting of stability conditions whose central charge $Z\colon K_0(\D(Q,I))\to \bC$ factors via this map.

By an action of a group $G$ on a quiver $Q$ we mean an action of $G$  on the sets $Q_0$ and $Q_1$ such that the maps $s$ and $t$ are $G$-equivariant.
There is an obvious induced action of $G$ on the path algebra $\bC Q$ and we always assume that  this action preserves the relations, i.e. that $g(I)= I$ for all $g\in G$. We then get an action of $G$ on the algebra $A=\bC Q/I$. We define the category $\rep^G(Q,I)$ of $G$-equivariant representations of $(Q,I)$ to be the category of finite-dimensional $G$-equivariant left $A$-modules. We define $\D^G(Q,I)=D^b(\rep^G(Q,I))$ to be its bounded derived category.

From now on we shall  always assume that $G$ is  finite and abelian, and  that the action on the set $Q_0$ is free. It follows that the action on $Q_1$ is also free. The quotient quiver $Q'=Q/G$ has  vertices $Q_0'=Q_0/G$, and arrows $Q_1'=Q_1/G$. We denote by $\pi\colon Q_0\to Q_0/G$ and $\pi\colon Q_1\to Q_1/G$ the quotient maps. We define  $I'\subset \bC Q'$ to be the image of $I$ under the quotient morphism $\pi\colon \mathbb{C}Q \to \mathbb{C}Q'$. There is a canonical functor
\begin{equation}
    \label{fun}
\Psi\colon \rep(Q',I')\to \rep^G(Q,I),\end{equation}
defined as follows. Given a representation $V'$ of the quiver $Q'$ consisting of vector spaces $V'_{i'}$ for vertices $i'\in Q'_0$ and linear maps $\rho'_{a'}\colon V'_{s(a')}\to V'_{t(a')}$ for arrows $a'\in Q'_1$,  the representation $V=\Psi(V')$ of $Q$ has vector spaces $V_i=V'_{\pi(i)}$ and linear maps $\rho_a=\rho'_{\pi(a)}$. The $G$-structure on $V$ is defined by taking the linearisation isomorphism $\lambda_g(i)\colon V_{g(i)}\to V_i$ to be  the identity for every vertex $i\in Q_0$ and group element $g\in G$. We leave the reader to extend the functor to morphisms and check functoriality.

The functor $\Psi$ is an equivalence of categories.
But  writing down a quasi-inverse requires making choices: different choices lead to isomorphic but strictly different functors. The required choice is a section $c\colon Q'_0\to Q_0$ of the quotient map $\pi\colon Q_0\to Q'_0$. That is, for each vertex $i'\in Q'_0$ we must choose a vertex $c(i')\in Q_0$ such that $\pi(c(i'))=i$.  We then also get a section $c\colon Q'_1\to Q_1$ by mapping an arrow $a'\in Q'_1$ to the unique arrow $a\in Q_1$ such that $\pi(a)=a'$ and $s(a)=c(s(a'))$.

Given the choice of the section  $c\colon Q'_0\to Q_0$ we  define a  quasi-inverse 
\begin{equation}
    \label{fun2}
\Upsilon\colon \rep^G(Q,I)\to \rep(Q',I')\end{equation} to the functor \eqref{fun} as follows. Suppose given a $G$-equivariant representation $V$ of $Q$ consisting of vector spaces $V_i$ for vertices $i\in Q_0$, linear maps $\rho_a\colon V_{s(a)}\to V_{t(a)}$ for arrows $a\in Q_1$, and linearisation isomorphisms $\lambda_g(i)\colon V_{g(i)}\to V_{i}$ for vertices $i\in Q_0$ and elements $g\in G$. The representation $V'=\Upsilon(V)$ of $Q'$ then has vector spaces $V'_{i'}=V_{c(i')}$ and linear maps $\rho'_{a'}=\lambda_{g(a')}(c(t(a')))\circ \rho_{c(a')}$, where $g(a')\in G$ is uniquely defined by the condition that $g(a')\cdot c(t(a'))=t(c(a'))$. Again we leave the reader to extend the functor to morphisms and check functoriality, and also to check that $\Upsilon$ is a quasi-inverse to $\Psi$.

\subsection{Inducing stability conditions}
Let us consider a finite abelian group $G$ acting on a quiver with relations $(Q,I)$ as above. 
There is a natural action of the group of characters $H=\Hom_{\bZ}(G,\bC^*)$ on the category $\rep^G(Q,I)$ and this can be transferred to an action of $H$ on the category $\rep(Q',I')$ via the  equivalence \eqref{fun}. Note that defining this action requires choosing a quasi-inverse to $\Psi$ and hence depends on the choice of section $c\colon Q'_0\to Q_0$ considered above. 

Concretely the  action of $H$ on the category $\rep(Q',I')$ is obtained as follows. As above, the choice of section $c$ gives rise to a labelling of every  arrow $a'\in Q'_1$ by an element $g(a')\in G$ uniquely defined by the condition that $g(a')\cdot c(t(a'))=t(c(a'))$.  Given a representation $V'$ of the quiver $Q'$ consisting of vector spaces $V'_{i'}$ for vertices $i'\in Q'_0$ and linear maps $\rho'_{a'}$ for arrows $a'\in Q'_1$, a character $\chi\in H$ then acts on $V'$ by leaving the vector spaces $V'_{i'}$ unchanged, and rescaling the linear map $\rho'_{a'}$ by the element  $\chi(g(a'))$.

Composing the derived functor associated to \eqref{fun} with the forgetful functor gives a functor
\begin{equation}
    \label{th'}
\Theta\colon \D(Q',I')\to \D(Q,I).\end{equation}
Using this to induce stability conditions leads to the following result.

\begin{thm}
\label{mainalg}
   In the above situation there is
an isomorphism of complex manifolds
\[F\colon \Stab(Q,I)^{G}\to \Stab(Q',I')^{H}.\]
 Moreover, given a stability condition $\sigma\in \Stab(Q,I)^G$
 
 \begin{itemize}
     \item [(i)]  an object $V\in \D(Q',I')$ is semistable of phase $\phi\in \bR$ for the stability condition $F(\sigma)$  precisely if $\Theta(V)\in \D(Q,I)$ is semistable of phase $\phi$ for the stability condition $\sigma$,
     \item [(ii)] the central charge of an object $V\in \D(Q',I')$ in the stability condition $F(\sigma)$ is equal to the central charge of the object $\Theta(V)\in \D(Q,I)$ in the stability condition $\sigma$. 
 \end{itemize}
\end{thm}

\begin{proof}
 Applying the result of Dell \cite[Lemma 2.23]{Dell} with $G=G$ and $\cD=D(Q,I)$ shows that  the forgetful functor $\Forg\colon \D^G(Q,I)\to \D(Q,I)$ induces a continuous bijection
\begin{equation}
    \label{hann}
F\colon \Stab(D(Q,I))^{G}\to \Stab(D^G(Q,I))^{H}.\end{equation}
Combining this with  the equivalence \eqref{fun} immediately gives a continuous bijection
\begin{equation*}F\colon \Stab(D(Q,I))^{G}\to \Stab(D(Q',I'))^{H}\end{equation*}
satisfying conditions (i) and (ii) of the statement.

At the level of central charges this bijection  arises from an isomorphism  of vector spaces
\begin{equation}
    \label{done?2}
K_0(\D(Q,I))^G\tensor_{\bZ}\bQ\isom K_0(\D(Q',I'))^H\tensor_{\bZ} \bQ\end{equation} induced by the functor \eqref{th'}. Given the simple form of $\Theta$ it is then easy to see directly that this descends to an isomorphism $(\bZ^{Q_0})^G\tensor_{\bZ} \bQ\isom (\bZ^{Q'_0})^H\tensor_{\bZ} \bQ$ at the level of dimension vectors. The rest of the argument proceeds as before.
\end{proof}

We give one example of Theorem \ref{mainalg} here. It is essentially the same as  Example \ref{egg} but treated algebraically. 

\begin{example}\label{egg2}
Take $Z=\bP^2$ and let $X=\omega_Z$ be the corresponding local CY threefold. It is well-known that there is a derived equivalence $\D_c(X)\isom \D(Q,I)$, where the quiver $Q$ is shown in  Figure \ref{fig:P2}(a), and the ideal $I$ is generated by commuting relations. The group $G=\bZ_3$ acts on $Q$ by rotations in the obvious way,  and this action preserves the relations $I$. The quotient quiver $Q'=Q/G$  is shown in Figure \ref{fig:P2}(b), and so the relevant algebra  $\bC Q'/I'\isom \bC[x,y,z]$ is the polynomial ring in three variables. Setting $Y=\bC^3$,  there is an equivalence $\D(Q',I')\isom \D_0(Y)$, and  it follows easily that $\Stab(Q',I')=\Stab(Y)$. As in Example \ref{egg}  there is a connected component $\Stab_0(Q',I')=\Stab(Q',I')$, which is isomorphic to $\bC$. Moreover the induced action of  $H=\Hom_{\bZ}(G,\bC^*)$ is trivial. Theorem \ref{mainalg} then gives a connected component $\Stab_0(Q,I)^G\subset \Stab(Q,I)^G$  and an identification $\Stab_0(Q,I)^G\isom \bC$.
\end{example}

\begin{figure}[tp]
% ! doesn't do what you think it does
\centering
\begin{subfigure}[m]{0.3\linewidth}
    \centering
    \includegraphics[width=\linewidth]{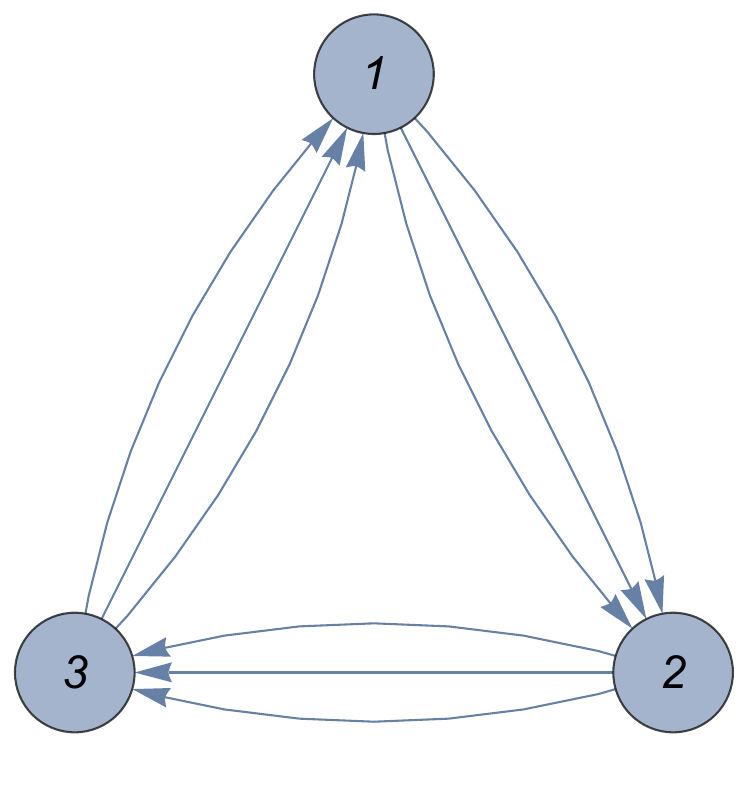}
    \caption{Quiver for $\omega_{\mathbb{P}^2}$.}\label{fig:DimerPdP5}
\end{subfigure}\hfil% equal to outside spacing
\begin{subfigure}[m]{0.3\linewidth}
\includegraphics[width=\linewidth]{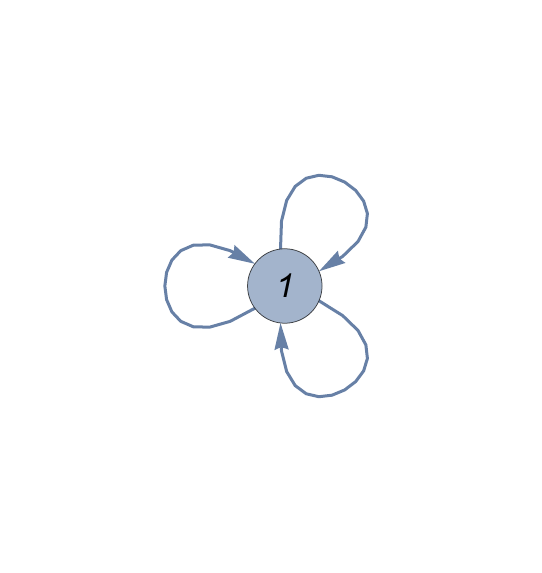}
 \caption{Quiver for $\mathbb{C}^3$}
 \end{subfigure}
 \caption{Quivers for Example \ref{egg2}}\label{fig:P2}
\end{figure}

\subsection{Brane tilings}\label{sec:brane}

 Brane tilings are a useful tool for constructing quivers with potential  whose Jacobi algebras are noncommutative crepant resolutions of  toric varieties. Their significance in this context was first recognized in the physics literature \cite{Hanany2005,Franco2006}. These developments, together  with later advances, were subsequently formalised mathematically in \cite{Broom}. 
 
 A brane tiling   is a   bipartite graph  on $\bR^2$  which is invariant under translations by a lattice $\Gamma=\bZ \cdot \omega_1\oplus \bZ\cdot \omega_2 \subset \bR^2$. Bipartite means that  the vertices  are of two colours, say black and white, and  each edge connects vertices of different colours. It is convenient to view a brane tiling as a bipartite graph  $\Lambda$  on the quotient torus $T=\mathbb R^2/\Gamma$.
 
 A brane tiling  $\Lambda$   has an associated  quiver with potential $(Q,W)$. The  quiver $Q$ is the dual graph to $\Lambda$, oriented so that  arrows go clockwise around white vertices and counterclockwise around black vertices. The  potential $W$ is  the signed sum of the cycles around the vertices  of $\Lambda$, with positive signs for black vertices and  negative signs for white vertices.
 
A perfect matching on a brane tiling $\Lambda$ is a subset of the edges of $\Lambda$  such that each vertex of $\Lambda$ lies on exactly one edge in this subset. Taking the difference of two perfect matchings defines a homology class in $H_1(T,\mathbb Z)$, so that fixing a reference perfect matching one gets a set of points, whose convex hull defines a lattice polygon $V\subset H_1(T,\mathbb R)\isom \mathbb R^2$. A practical way to compute the polygon $V$ is to calculate the determinant of the Kasteleyn matrix of the brane tiling \cite{Hanany2005}.
 
 The work \cite{Broom} gives a sufficient condition, called geometric consistency, for the Jacobi algebra of $(Q,W)$ to define a noncommutative crepant resolution of the toric Calabi-Yau  threefold associated to the cone over the lattice polygon $V$. One considers yet another graph, the quad graph, whose vertices are given by the union of the vertices of the quiver $Q$ and of the brane tiling $\Lambda$, and which has an edge connecting a vertex $e$ of $Q$ to a vertex $v$ of $\Lambda$  if and only if the face of $\Lambda$ corresponding to $e$ has $v$ as a vertex. This  quad graph defines a tiling of the torus whose faces are all quadrilateral. The brane tiling is said to be geometrically consistent if the quad graph admits an embedding in $T$ such that all  edges have the same length.
 
 The following is a combination of Theorem 1.1 and Theorem 8.5 from \cite{Broom}.
 \begin{thm}
     Given a geometrically consistent brane tiling $\Lambda$, the Jacobi algebra of the associated quiver with potential $(Q,W)$ is a noncommutative crepant resolution of the derived category of the toric variety associated to the lattice polygon $V$.
 \end{thm}

 All the brane tilings appearing below can easily be checked to be geometrically consistent, and so the above result applies. To put ourselves in the context of Theorem \ref{mainalg} we will consider a brane tiling $\Lambda$ which is preserved by the action of a finite abelian group $G$ on the torus $T=\bR^2/\Gamma$. 

In most of our examples, the group $G$ will act via translation by a subgroup of $T$.
 More concretely, we can consider lattices $\Gamma \subset \Gamma' \subset \mathbb R^2$ and a bipartite graph  on $\bR^2$ which is invariant under translations by the lattice  $\Gamma'$. Then the quiver with potential $(Q,W)$ associated to the  induced brane tiling on the torus $\mathbb R^2/\Gamma$ carries an action of the group $G=\Gamma'/\Gamma$. Moreover  the quotient quiver with potential $(Q',W')$ is the one associated to the induced  brane tiling on the torus $\mathbb R^2/\Gamma'$.

\section{Main example: local $\bP^1\times \bP^1$}

In this section we present our main example relating to the local Calabi-Yau threefold $\omega_Z$ with $Z=\bP^1\times \bP^1$. After constructing the relevant invariant stability conditions, we compute the corresponding DT invariants.

\subsection{Stability conditions on the resolved conifold}

Let $Y$ be the total space of the bundle $\O_{\bP^1}(-1)^{\oplus 2}$. This is a quasi-projective Calabi-Yau threefold known as the resolved conifold: contracting the zero section gives the threefold ordinary double point $(xy-zw=0)\subset \bC^4$.
The variety $Y$ contains a unique compact curve, namely the zero section $C\subset Y$, which  defines a fundamental class $\beta=[C]\in H_2(Y,\bZ)$. Let $\delta=[y]\in H_0(Y,\bZ)$ be the fundamental class of a point $y\in Y$. Then
\[\Gamma:=H_*(Y,\bZ)=H_2(Y,\bZ)\oplus H_0(Y,\bZ)=\bZ\cdot \beta \oplus \bZ \cdot \delta.\]
The Chern character map together with Poincar{\'e} duality $H^*_{c}(Y,\bZ)\isom H_*(Y,\bZ)$ gives a group homomorphism
\[\ch=(\ch_2,\ch_3)\colon K_0(\D(Y)) \to H_*(Y,\bZ).\]
A stability condition on $\D(Y)$ lies in the subspace $\Stab(Y)$ precisely if its central charge factors via this map. 
We write $\O_C(n)$ for the degree $n$ line bundle supported on the rational curve $C\subset Y$, and $\O_y$ for the skyscraper sheaf supported at a point $y\in Y$. Then
\[\ch(\O_C(n))=\beta+n\delta, \qquad \ch(\O_y)=\delta.\]

The standard t-structure on $\D(Y)$ restricts to give a bounded t-structure on the category $\D_c(Y)$ whose heart can be identified with $\Coh_{c}(Y)\ \D(Y)$.
We denote by $\Stab_0(Y)\subset \Stab(Y)$ the connected component containing stability conditions with this heart. Sending a stability condition to its central charge defines a map
\[\varpi\colon \Stab_0(Y)\to \Hom_{\bZ}(\Gamma,\bC)\]
The following result is proved by the methods of \cite{Toda}, see particularly the Example following Theorem 5.3.

\begin{thm}
\label{conifold}
For any stability condition in $\Stab_0(Y)$ the semistable objects consist of shifts of the following objects
    \begin{itemize}
    \item[(i)] coherent sheaves of the form $\O_C(n)\tensor_{\bC} V$ for $n\in \bZ$ and $V$ a vector space;
        \item[(ii)]zero-dimensional coherent sheaves.
    \end{itemize}
    Moreover the map $\varpi$ is a regular covering map over its image, which is the open subset
\[\{Z\colon \Gamma\to \bC: Z(\beta+n\delta)\neq 0, Z(\delta)\neq 0\}.\]
\end{thm}

We obtain the familiar picture in Figure~\ref{fig:raydiag} of a countable set of active  rays in $\bC$ spanned by the points $\pm Z(\O_C(n))$ for $n\in \bZ$, each containing a single stable object up to shift, together with a pair of limiting rays $\pm Z(\O_x)$ containing all zero-dimensional sheaves and their shifts.

\subsection{Invariant stability conditions} \label{sec:StablesF0}
 \begin{figure}[h!]% ! doesn't do what you think it does
\centering
\begin{subfigure}[m]{0.3\linewidth}
    \centering
    \includegraphics[width=\linewidth]{Figures/QuiverF0.jpg}
    \caption[b]{$Q_{\omega_{\mathbb{P}^1\times\mathbb{P}^1}}$}\label{fig:QuiverP1P1}
 \end{subfigure}\qquad\qquad\qquad
 \begin{subfigure}[m]{0.3\linewidth}
    \centering
    \includegraphics[width=\linewidth]{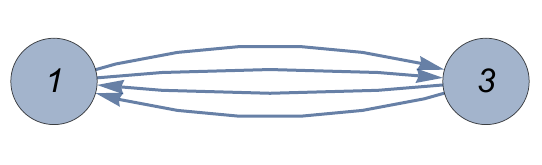}
    \caption[b]{$Q_{\omega_{\mathbb{P}^1\times\mathbb{P}^1}/\mathbb{Z}_2}=Q_{\O_{\mathbb{P}^1}(-1)^{\oplus 2}}$}\label{fig:QuiverPdP5Z2Z2}
 \end{subfigure}
 \caption{}\label{fig:P1P1}
\end{figure}
 Let $H=\bZ_2$ act on the resolved conifold $Y=\O_{\bP^1}(-1)^{\oplus 2}$ by multiplication by $-1$ in the fibres of the bundle. The quotient $Y/H$ has a unique crepant resolution which is isomorphic to the local  threefold  $X=\omega_Z$ where $Z=\bP^1\times \bP^1$.
 
 The group $H$ acts trivially on $\Gamma=H_*(Y,\bZ)$ and preserves the class of semistable objects listed in Theorem~\ref{conifold}. This implies that the induced action of $H$ on $\Stab_0(Y)$ is trivial. The map $F$ of Theorem~\ref{main} defines an embedding 
$\Stab_0(Y)\hookrightarrow\Stab(X)^{G}$.
We refer to the points in the image of this map as invariant stability conditions. 
Denoting by $\Psi\colon \D^H(Y)\to \D(X)$ the inverse of the equivalence \eqref{phi}, we obtain

\begin{thm}
\label{chee}
     Let $\sigma\in \Stab_0(X)^G$ be an invariant stability condition. Then the semistable objects consist of the images under $\Psi$ of the following objects and their shifts
     \begin{itemize}
      \item[(i)] $H$-equivariant coherent sheaves of the form $\O_C(n)\tensor_{\bC} V$ for $n\in \bZ$ and  $V$  a  representation of $H$,
        \item[(ii)]zero-dimensional $H$-equivariant coherent sheaves on $Y$.
    \end{itemize}
\end{thm}

The collection of rays is as before: there are two opposite rays for each $n\in \bZ$ appearing in (i) and two limiting rays containing the shifts of the objects in (ii). 
Let $\chi_0,\chi_1$ denote the trivial and non-trivial characters of $H$. Define classes
\[\gamma_1=\ch(\Psi(\O_C\tensor \chi_0)),  \qquad \gamma_2=\ch(\Psi(\O_C(-1)[1]\tensor \chi_0)),\]\[ \gamma_3=\ch(\Psi(\O_C\tensor \chi_1)),  \qquad \gamma_4=\ch(\Psi(\O_C(-1)[1]\tensor \chi_1)).\]
The skyscraper sheaf of a point $y\in C$ fits into a triangle
\[\O_C\to \O_y\to \O_C(-1)[1],\]
which implies that
\[\ch(\Psi(\O_y\tensor \chi_0))=\gamma_1+\gamma_2, \qquad \ch(\Psi(\O_y\tensor \chi_1))=\gamma_3+\gamma_4.\]
Since a point $y\in Y$ is the image under $\Psi$ of a free orbit of $H$ on $Y$ we find that
\[\ch(\O_y)= \delta:=\gamma_1+\gamma_2+\gamma_3+\gamma_4.\]
 The objects in Theorem \ref{chee} (i) have class $p((n+1)\gamma_1+n\gamma_2)+q((n+1)\gamma_3+n\gamma_4)$ where $V=\chi_0^{\oplus p}\oplus \chi_1^{\oplus q}$. The objects in (ii) have classes of the form $s(\gamma_1+\gamma_2)+t(\gamma_3+\gamma_4)$.

\subsection{The DT invariants} \label{sec:DTF0}

We can now compute the DT invariants for the above invariant stability conditions. The only tricky ray is the central one corresponding to zero-dimensional equivariant sheaves on $Y$, but the DT invariants for this ray were computed by Bryan, Cadman and Young \cite{BCY} using the orbifold topological vertex. 

\begin{thm}\label{thm:DTP1P1}
Let $\sigma\in \Stab(X)^G$  be an invariant stability condition as above.  Then  the nonzero DT invariants  are
    \[\Omega((n+1)\gamma_1+n\gamma_2)= \Omega(n\gamma_1+(n+1)\gamma_2)=1, \qquad n\in \bZ\]\[\Omega((n+1)\gamma_3+n\gamma_4)=\Omega(n\gamma_3+(n+1)\gamma_4)=1,\qquad n\in \bZ\]
    \[\Omega((\gamma_1+\gamma_2)+n\delta)=\Omega ((\gamma_3+\gamma_4)+n\delta)=-2, \qquad n\in \bZ,\]\[ \Omega( n\delta)=-4, \qquad n\in \bZ\setminus\{0\}.\]
\end{thm}

\begin{proof}
For any ray except the central one there is an $n\in \bZ$ such that the ray contains exactly two stable objects $\Psi(\O_C(n)\tensor \chi_p)$ up to shift where $p\in \{1,2\}$. These objects are spherical and there are no extensions between them. This implies that the relevant DT invariants are $+1$.

The category of semistable objects on the central ray is equivalent to the category of zero-dimensional sheaves on the orbifold $\cY=[Y/H]$. The ideal sheaf DT invariants for this orbifold were computed in  \cite{BCY}, see particularly Section 4.3 and Theorem 12. The degree 0 invariants  take the form
\[\underline{\DT}_0(\cY)=\prod_{m\geq 1} (1-q_0^{m} q_1^{m-1})^{-2m}\cdot (1-q_0^{m} q_1^{m+1})^{-2m}\cdot (1-q_0^m q_1^m)^{-4m},\]
where $q_0,q_1$ are  formal variables corresponding to the classes of the equivariant sheaves $\O_y\tensor \chi_0$ and $\O_y\tensor \chi_1$. The stated DT  invariants are obtained from the exponents appearing in this expression  exactly as in \cite[Sections 7.5.2--7.5.4]{JS}.
\end{proof}

\section{Other examples}\label{sec:ex}

The above arguments can be applied to several other cases.

\subsection{(Pseudo) $dP_5$}

The singular toric $CY_3$ with toric diagram as in Figure \ref{fig:PdP5Quotient}\, (\subref{fig:PdP5})
has a resolution $\omega_{PdP_5}$ known as local Pseudo del Pezzo 5 toric threefold. Here $PdP_5$ denotes the toric weak Fano surface obtained by blowing up $\mathbb{P}^2$ at five non-generic points, and as is clear from the toric diagrams in Figure \ref{fig:PdP5Quotient}, the singular Calabi-Yau threefold is a $\mathbb{Z}_2\times\mathbb{Z}_2$ orbifold of the conifold.  The relevant brane tilings are discussed below, and can be found in Figure \ref{fig:dP5Dimers}.
The corresponding quiver with potential $(Q_{PdP_5},W_{PdP_5})$ can be obtained from either brane tiling techniques (model 4a in \cite{Hanany2012}) or from exceptional collections \cite{Beaujard2020}. The quiver is depicted in Figure \ref{fig:dP5Quivers}(a). Note that  the same quiver arises for the generic local del Pezzo 5 threefold, obtained from the blowup of $\mathbb{P}^2$ at five general points (see Section C.7 in \cite{Beaujard2020}), although with a different potential.

The potential for $\omega_{PdP_5}$ is
\begin{gather}
    W_{PdP_5}=-X_{1,3} X_{3,5} X_{5,7} X_{7,1}+X_{1,4} X_{4,6} X_{6,7} X_{7,1}+X_{2,4} X_{4,5} X_{5,7} X_{7,2}-X_{2,3} X_{3,6} X_{6,7} X_{7,2}\nonumber\\
    -X_{1,4} X_{4,5} X_{5,8} X_{8,1}+X_{1,3} X_{3,6} X_{6,8} X_{8,1}+X_{2,3} X_{3,5} X_{5,8} X_{8,2}-X_{2,4} X_{4,6} X_{6,8} X_{8,2}.
\end{gather}
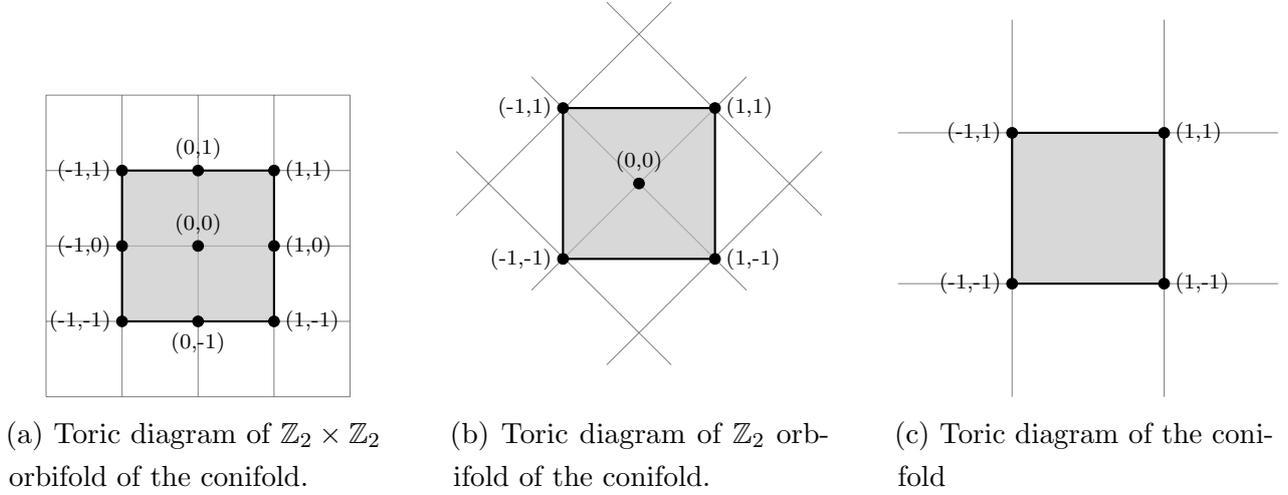
\begin{figure}[t]
    \centering
    % First subfigure
    \begin{subfigure}[b]{0.28\textwidth}
        \centering
\begin{tikzpicture}
    % Grid style
    \draw[very thin, gray] (-2,-2) grid (2,2);
    
    % Square boundary
    \draw[fill=lightgray, fill opacity=0.6, draw=black, thick] (-1,1) -- (1,1) -- (1,-1) -- (-1,-1) -- cycle;
    
    % Vertices of the square
    \filldraw[black,] (-1,1) circle (2pt) node[anchor=east] {\tiny(-1,1)};
    \filldraw[black,] (1,1) circle (2pt) node[anchor=west] {\tiny(1,1)};
    \filldraw[black,] (1,-1) circle (2pt) node[anchor=west] {\tiny(1,-1)};
    \filldraw[black,] (-1,-1) circle (2pt) node[anchor=east] {\tiny(-1,-1)};
    
    % Midpoints on the edges
    \filldraw[black] (0,1) circle (2pt) node[anchor=south] {\tiny(0,1)};
    \filldraw[black] (1,0) circle (2pt) node[anchor=west] {\tiny(1,0)};
    \filldraw[black] (0,-1) circle (2pt) node[anchor=north] {\tiny(0,-1)};
    \filldraw[black] (-1,0) circle (2pt) node[anchor=east] {\tiny(-1,0)};
    
    % Center point
    \filldraw[black] (0,0) circle (2pt) node[anchor=south] {\tiny(0,0)};
\end{tikzpicture}
        \caption{Toric diagram of $\mathbb{Z}_2\times\mathbb{Z}_2$ orbifold of the conifold.}\label{fig:PdP5}
    \end{subfigure}
    \qquad
    % Second subfigure
    \begin{subfigure}[b]{0.28\textwidth}
        \centering
\begin{tikzpicture}
    % Grid style
    \draw[very thin, gray,rotate=45] (-2,-2) grid[step=1.4] (2,2);
    
    % Square boundary
    \draw[fill=lightgray, fill opacity=0.6, draw=black, thick] (-1,1) -- (1,1) -- (1,-1) -- (-1,-1) -- cycle;
    
    % Vertices of the square
    \filldraw[black,] (-1,1) circle (2pt) node[anchor=east] {\tiny(-1,1)};
    \filldraw[black,] (1,1) circle (2pt) node[anchor=west] {\tiny(1,1)};
    \filldraw[black,] (1,-1) circle (2pt) node[anchor=west] {\tiny(1,-1)};
    \filldraw[black,] (-1,-1) circle (2pt) node[anchor=east] {\tiny(-1,-1)};
    
    \filldraw[black] (0,0) circle (2pt) node[anchor=south] {\tiny(0,0)};
\end{tikzpicture}
        \caption{Toric diagram of $\mathbb{Z}_2$ orbifold of the conifold.}\label{fig:p1p1}
    \end{subfigure}
    \qquad
    % Third subfigure
    \begin{subfigure}[b]{0.28\textwidth}
        \centering
 \begin{tikzpicture}
    % Grid style
    \draw[very thin, gray,shift={(1,1)}] (-3.5,-3.5) grid[step=2] (1.5,1.5);
    
    % Square boundary
    \draw[fill=lightgray, fill opacity=0.6, draw=black, thick] (-1,1) -- (1,1) -- (1,-1) -- (-1,-1) -- cycle;
    
    % Vertices of the square
    \filldraw[black,] (-1,1) circle (2pt) node[anchor=east] {\tiny(-1,1)};
    \filldraw[black,] (1,1) circle (2pt) node[anchor=west] {\tiny(1,1)};
    \filldraw[black,] (1,-1) circle (2pt) node[anchor=west] {\tiny(1,-1)};
    \filldraw[black,] (-1,-1) circle (2pt) node[anchor=east] {\tiny(-1,-1)};

\end{tikzpicture}
        \caption{Toric diagram of the conifold}\label{fig:conifold}
    \end{subfigure}
    \caption{Fans arising from orbifolding $PdP_5$.}\label{fig:PdP5Quotient}
\end{figure}
\begin{figure}[b]% ! doesn't do what you think it does
\centering
\begin{subfigure}[m]{0.3\linewidth}
    \centering
    \includegraphics[width=\linewidth]{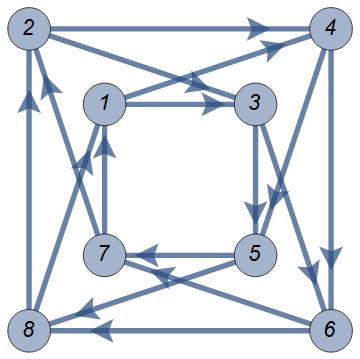}
    \caption{$Q_{PdP_5}$.}\label{fig:QuiverPdP5}
\end{subfigure}\hfil% equal to outside spacing
\begin{subfigure}[m]{0.3\linewidth}
    \centering
    \includegraphics[width=\linewidth]{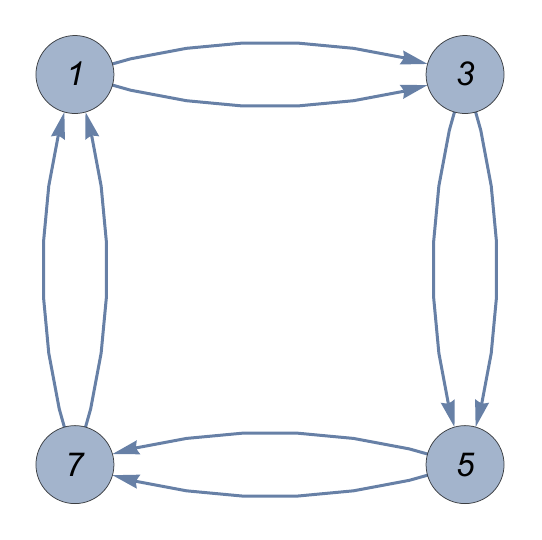}
    \caption[b]{$Q_{PdP_5}/\mathbb{Z}_2=Q_{{\mathbb{P}^1\times\mathbb{P}^1}}$}\label{fig:QuiverPdP5Z2}
 \end{subfigure}
 \begin{subfigure}[m]{0.3\linewidth}
    \centering
    \includegraphics[width=\linewidth]{Figures/dP5Z2Z2Quiver.pdf}
    \caption[b]{$Q_{PdP_5}/\mathbb{Z}_2\times\mathbb{Z}_2=Q_{con}$}\label{fig:QuiverPdP5Z2Z2}
 \end{subfigure}
 \caption{Quotient quivers for $\omega_{PdP5}$.}\label{fig:dP5Quivers}
\end{figure}
We will consider quotients of the quiver $Q_{PdP_5}$ by the $\mathbb{Z}_2\times\mathbb{Z}_2$ group generated by 
\begin{equation}
    \pi_{PdP_5}^{(1)}=(1,2)(3,4)(5,6)(7,8)\quad \pi_{PdP_5}^{(2)}=(1,5)(2,6)(3,7)(4,8),
\end{equation}
and by the $\mathbb{Z}_2$ group generated by only $\pi_{PdP_5}^{(1)}$. The quotient quivers are depicted in Figure \ref{fig:dP5Quivers}, and coincide respectively with the quiver of $\omega_{\mathbb{P}^1\times\mathbb{P}^1}$ and of the resolved conifold. The resulting potentials are
{\smaller\begin{gather}
    \frac{1}{2}W_{PdP_5/\mathbb{Z}_2}=\frac{1}{2}\bigg[-X_{1,3}^{(1)} X_{3,5}^{(1)} X_{5,7}^{(1)} X_{7,1}^{(1)}+X_{1,3}^{(2)} X_{3,5}^{(1)} X_{5,7}^{(2)} X_{7,1}^{(1)}+X_{1,3}^{(1)} X_{3,5}^{(2)} X_{5,7}^{(1)} X_{7,1}^{(2)}\bigg]\nonumber\\
    +\frac{1}{2}\bigg[-X_{1,3}^{(2)} X_{3,5}^{(2)} X_{5,7}^{(2)} X_{7,1}^{(2)}
    -X_{1,3}^{(2)} X_{3,5}^{(2)} X_{5,7}^{(2)} 
    X_{7,1}^{(2)}
    +X_{1,3}^{(1)} X_{3,5}^{(2)} X_{5,7}^{(1)} X_{7,1}^{(2)}
    +X_{1,3}^{(2)} X_{3,5}^{(1)} X_{5,7}^{(2)} X_{7,1}^{(1)}-X_{1,3}^{(1)} X_{3,5}^{(1)} X_{5,7}^{(1)} X_{7,1}^{(1)}\bigg] \nonumber \\
    = -X_{1,3}^{(1)} X_{3,5}^{(1)} X_{5,7}^{(1)}X_{7,1}^{(1)}-X_{1,3}^{(2)} X_{3,5}^{(2)} X_{5,7}^{(2)} X_{7,1}^{(2)}+X_{1,3}^{(2)} X_{3,5}^{(1)} X_{5,7}^{(2)} X_{7,1}^{(1)}+X_{1,3}^{(1)} X_{3,5}^{(2)} X_{5,7}^{(1)} X_{7,1}^{(2)}  
    = W_{\mathbb{P}^1\times\mathbb{P}^1}
\end{gather}}
for the $\mathbb{Z}_2$ quotient, and

{\small\begin{gather}
     \frac{1}{4}W_{PdP_5/\mathbb{Z}_2\times\mathbb{Z}_2}=\frac{1}{2}\left(-X_{1,3}^{(1)} X_{3,1}^{(1)} X_{1,3}^{(2)}X_{3,1}^{(2)}-X_{1,3}^{(2)} X_{3,1}^{(2)} X_{1,3}^{(1)} X_{3,1}^{(1)}+X_{1,3}^{(2)} X_{3,1}^{(1)} X_{1,3}^{(1)} X_{3,1}^{(2)}+X_{1,3}^{(1)} X_{3,1}^{(2)} X_{1,3}^{(2)} X_{3,1}^{(1)}\right)  \nonumber\\
    \mathop{\sim}_{\text{cyclic perm.}}-X_{1,3}^{(1)} X_{3,1}^{(1)} X_{1,3}^{(2)}X_{3,1}^{(2)}+X_{1,3}^{(2)} X_{3,1}^{(1)} X_{1,3}^{(1)} X_{3,1}^{(2)} =W_{\rm conifold}
\end{gather}}
for the $\mathbb{Z}_2\times\mathbb{Z}_2$ quotient. Here we divided the potential $W$ (which is only defined up to an overall scale) by the order of the group, in order to obtain a potential with all coefficients being $\pm1$.

\begin{figure}[h!]% ! doesn't do what you think it does
\centering
\begin{subfigure}[t]{0.3\linewidth}
    \centering
    \includegraphics[width=\linewidth]{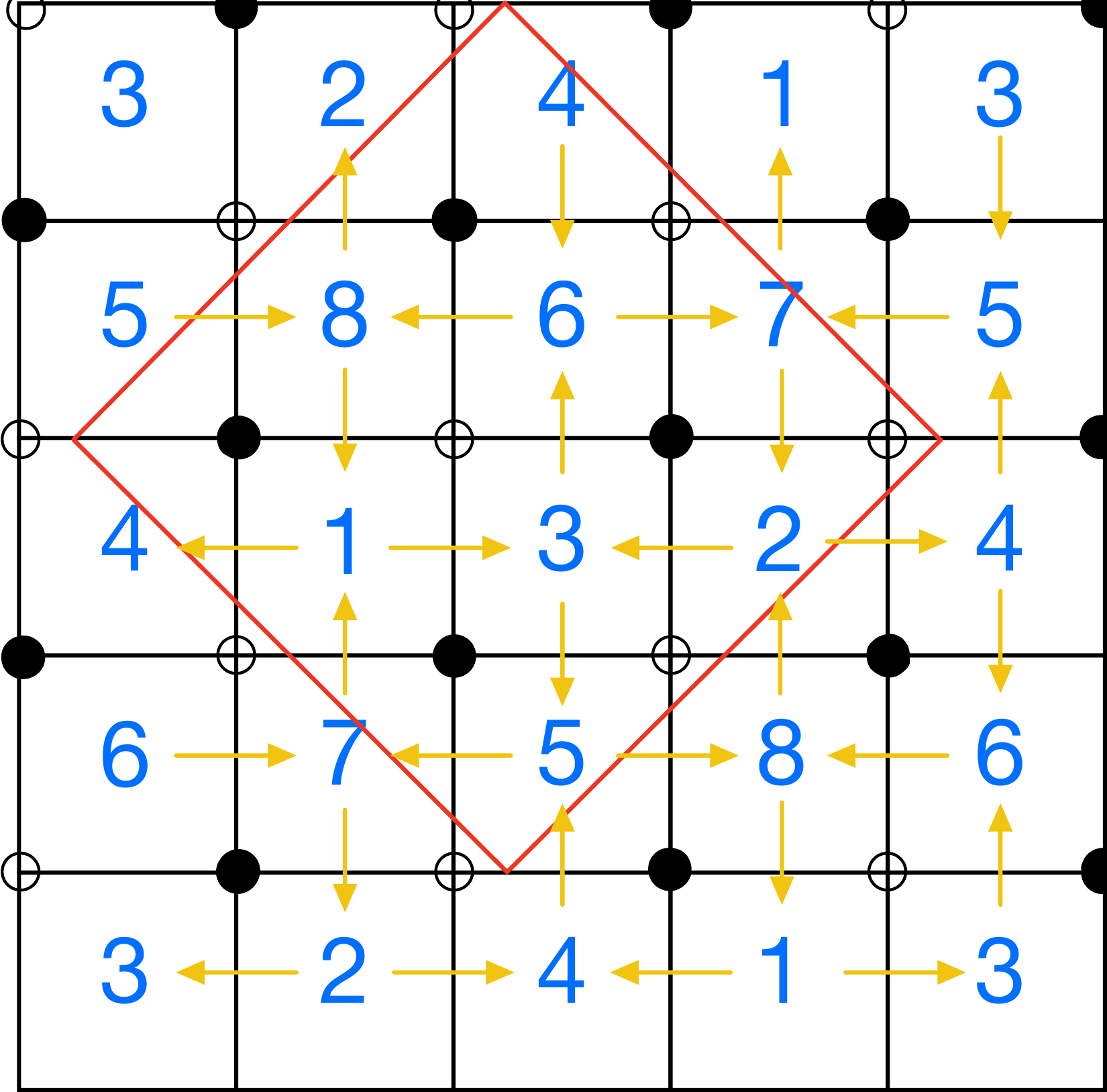}
    \caption{}\label{fig:DimerPdP5}
\end{subfigure}\hfil% equal to outside spacing
\begin{subfigure}[t]{0.3\linewidth}
    \centering
    \includegraphics[width=\linewidth]{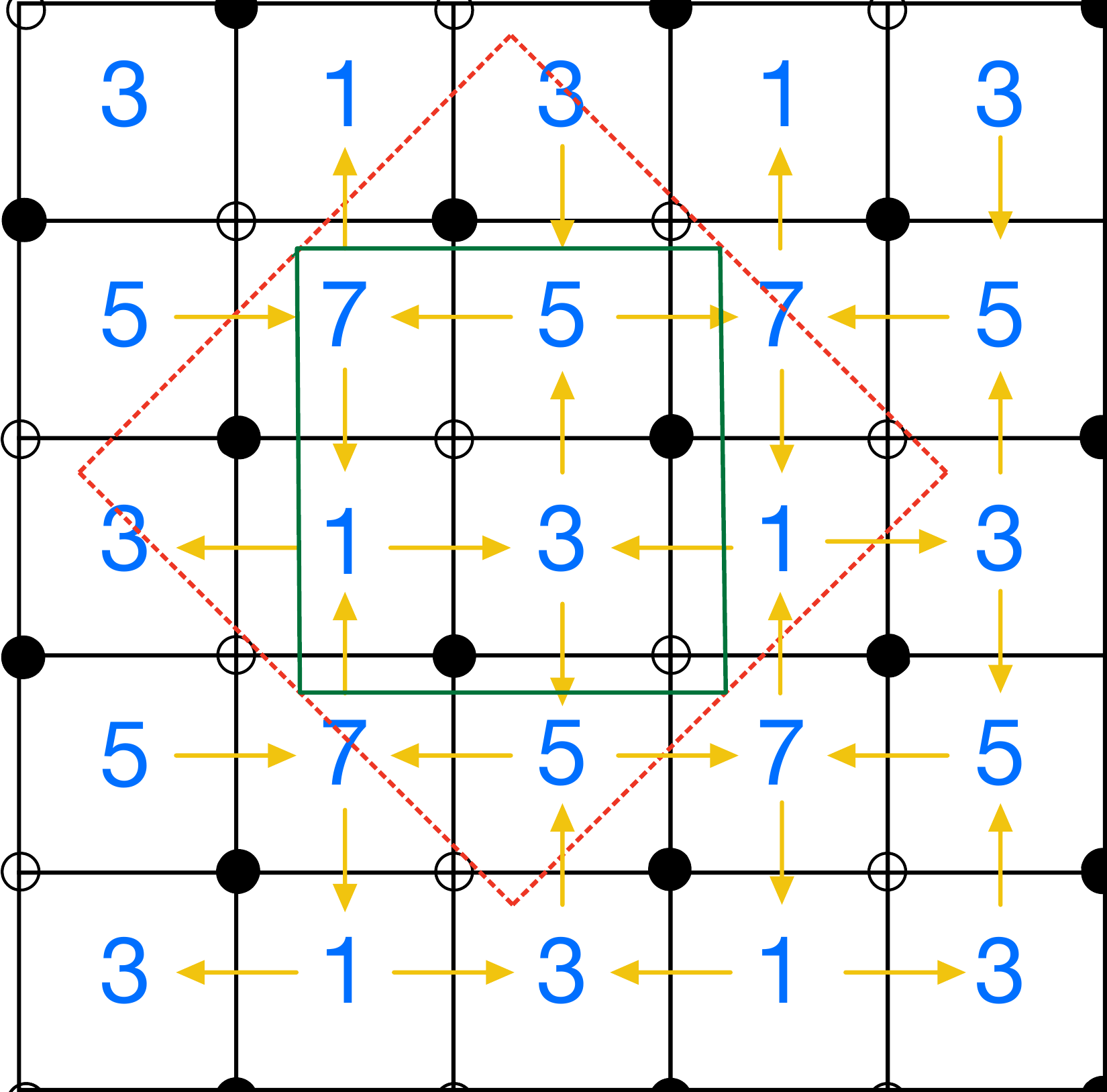}
    \caption[b]{}\label{fig:DimerPdP5Z2}
 \end{subfigure}\hfil
 \begin{subfigure}[t]{0.3\linewidth}
    \centering
    \includegraphics[width=\linewidth]{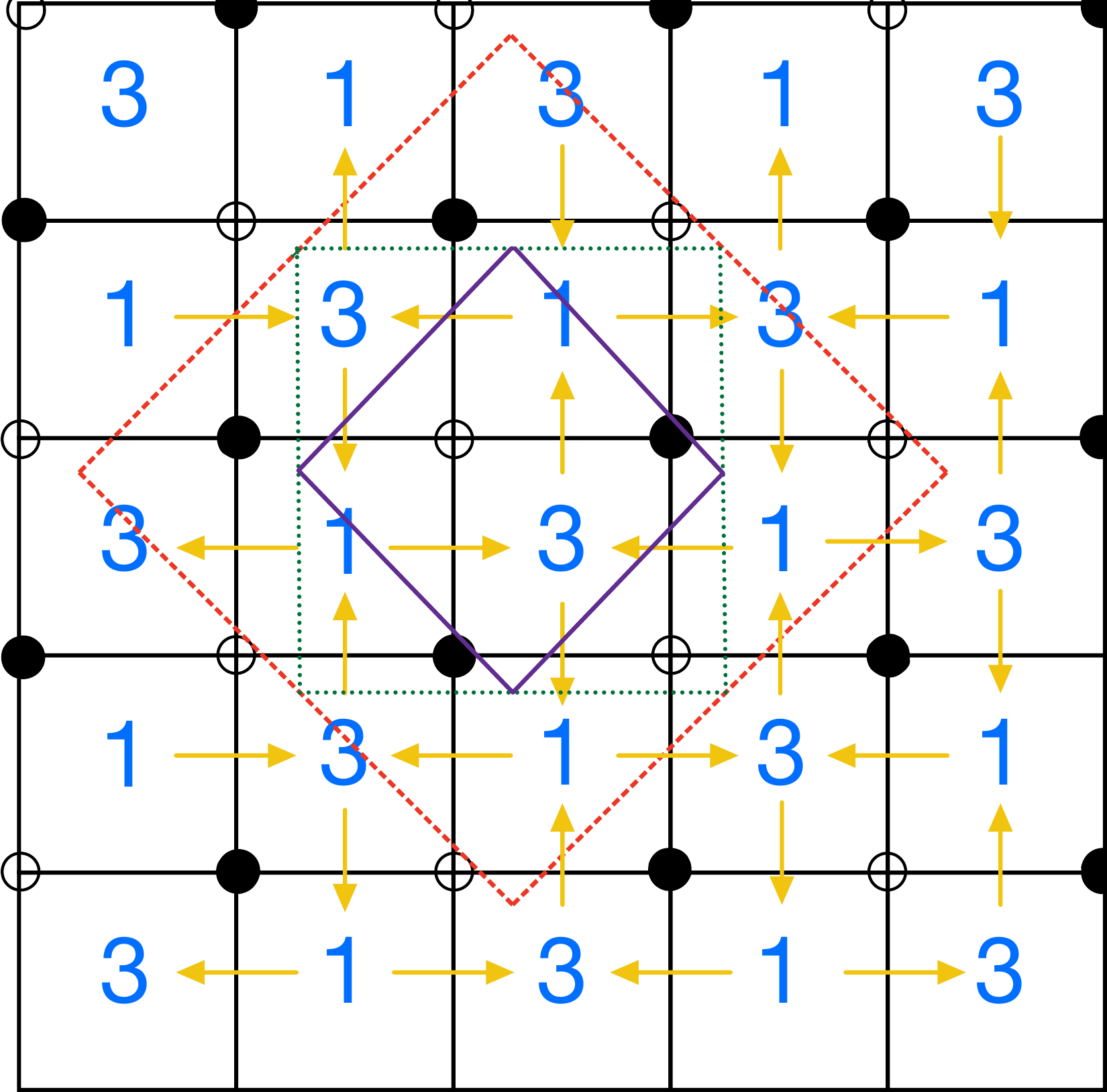}
    \caption[b]{}\label{fig:DimerZ2Z2}
 \end{subfigure}
 \caption{Dimer models and periodic quivers for $\omega_{PdP_5}$ (red fundamental domain), $\omega_{\mathbb{P}^1\times\mathbb{P}^1}$ (green fundamental domain), and the resolved conifold (purple fundamental domain).}\label{fig:dP5Dimers}
\end{figure}

The quotient operation described in this example are best described in terms of dimer models. We define a chain of index 2 subgroups $\Gamma_{con}\subset \Gamma_{\bP^1\times\bP^1}\subset \Gamma_{PdP_5}\subset \bZ^{\oplus 2}$ by
{\small\begin{equation}
    \label{label}
\Gamma_{PdP_5}=\bZ \left(\begin{array}{c} 2 \\  -2 \end{array} \right) \oplus \bZ \left(\begin{array}{c} 2 \\  2 \end{array} \right), \quad \Gamma_{\bP^1\times\bP^1}=\bZ  \left(\begin{array}{c} 2 \\  0 \end{array} \right)\oplus \bZ  \left(\begin{array}{c} 0 \\  2 \end{array} \right)\quad \Gamma_{con}=\bZ \left(\begin{array}{c} 1 \\  -1 \end{array} \right)\oplus \bZ \left(\begin{array}{c} 1 \\  1 \end{array} \right).\end{equation}}

The dimer models  for $\omega_{PdP_5}$, $\omega_{\bP^1\times\bP^1}$ and the resolved conifold respectively are shown in Figure \ref{fig:dP5Dimers}. These are dimer models on tori with lattices given respectively by $\Gamma_{PdP5}$ for $\omega_{PdP_5}$ (Figure \ref{fig:dP5Dimers}(\subref{fig:DimerPdP5})), by the lattice $\Gamma_{\mathbb{P}^1\times\mathbb{P}^1}$ for $\omega_{\mathbb{P}^1\times\mathbb{P}^1}$ (Figure \ref{fig:dP5Dimers}(\subref{fig:DimerPdP5Z2})) and by the lattice $\Gamma_{con}$ for the resolved conifold (Figure \ref{fig:dP5Dimers}(\subref{fig:DimerZ2Z2})). As we discussed in Section \ref{sec:brane}, the quotienting operation on the respective quivers from Figure \ref{fig:dP5Quivers} is realised on the brane tilings by extending the lattice from $\Gamma_{dP_5}$ to $\Gamma_{\mathbb{P}^1\times\mathbb{P}^1}$ or $\Gamma_{con}$.

\subsection{$dP_3$}
Consider the $dP_3$ surface obtained by the blowup of three general points in $\mathbb{P}^2$, and the local $CY_3$ $\omega_{dP_3}$, resolving the singularity given by the toric diagram in Figure~\ref{fig:DimerAndToricDiagramdP3}(\subref{fig:dP3}).
\begin{figure}
\centering
\begin{subfigure}[t]{0.4\linewidth}
        \centering
    \includegraphics[width=0.5\linewidth]{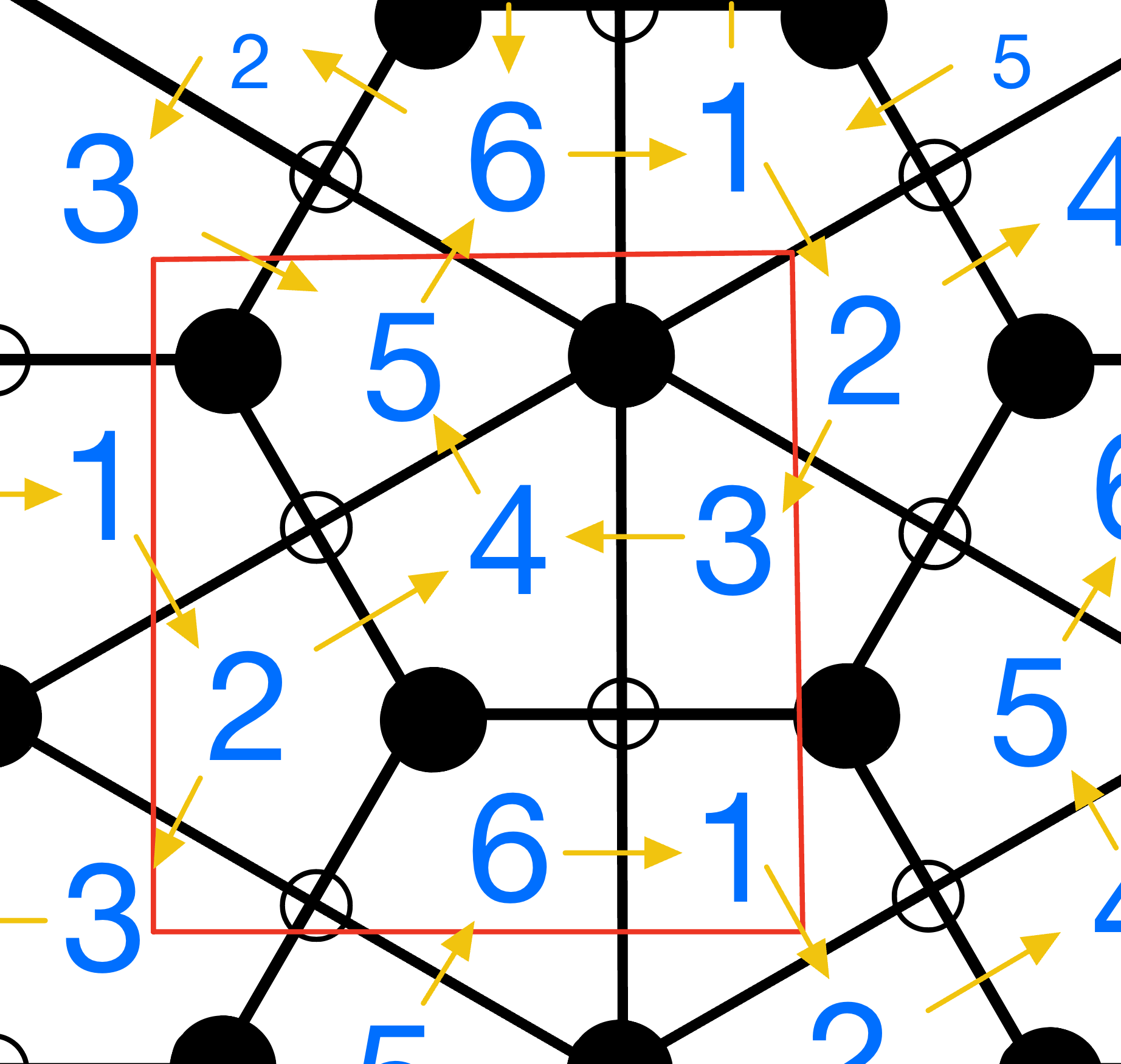}
    \caption{Dimer model for $\omega_{dP_3}$}
    \label{fig:dP3dimer}
\end{subfigure}
\begin{subfigure}[t]{0.4\linewidth}
\centering
\begin{tikzpicture}
    % Grid style
    \draw[very thin, gray] (-2,-2) grid (2,2);
    
    % Square boundary
    \draw[fill=lightgray, fill opacity=0.6, draw=black, thick] (-1,-1) -- (0,-1) -- (1,0) -- (1,1) -- (0,1) -- (-1,0) -- cycle;
    
    % Vertices of the square
    % \filldraw[black,] (-1,1) circle (2pt) node[anchor=east] {\tiny(-1,1)};
    \filldraw[black,] (1,1) circle (2pt) node[anchor=west] {\tiny(1,1)};
    % \filldraw[black,] (1,-1) circle (2pt) node[anchor=west] {\tiny(1,-1)};
    \filldraw[black,] (-1,-1) circle (2pt) node[anchor=east] {\tiny(-1,-1)};
    
    % Midpoints on the edges
    \filldraw[black] (0,1) circle (2pt) node[anchor=south] {\tiny(0,1)};
    \filldraw[black] (1,0) circle (2pt) node[anchor=west] {\tiny(1,0)};
    \filldraw[black] (0,-1) circle (2pt) node[anchor=north] {\tiny(0,-1)};
    \filldraw[black] (-1,0) circle (2pt) node[anchor=east] {\tiny(-1,0)};
    
    % Center point
    \filldraw[black] (0,0) circle (2pt) node[anchor=south] {\tiny(0,0)};
\end{tikzpicture}
\caption{Toric diagram of local $dP_3$.}\label{fig:dP3}
\end{subfigure}
\caption{}\label{fig:DimerAndToricDiagramdP3}
\end{figure}
The brane tiling in Figure \ref{fig:DimerAndToricDiagramdP3}(\subref{fig:dP3dimer}) (model 10a in \cite{Hanany2012})
gives rise to the quiver in Figure~\ref{fig:QuiversdP3}(\subref{fig:QuiverdP3})
with potential

\begin{gather}
    W_{dP_3}  = X_{12}X_{23}X_{34}X_{45}X_{56}X_{61}-X_{23}X_{35}X_{56}X_{62}-X_{13}X_{34}X_{46}X_{61}-X_{12}X_{24}X_{45}X_{51} \nonumber\\
    +X_{13}X_{35}X_{51}+X_{24}X_{46}X_{62}\,.
\end{gather}
\begin{figure}[h]% ! doesn't do what you think it does
\centering
\begin{subfigure}[t]{0.25\linewidth}
    \centering
    \includegraphics[width=\linewidth]{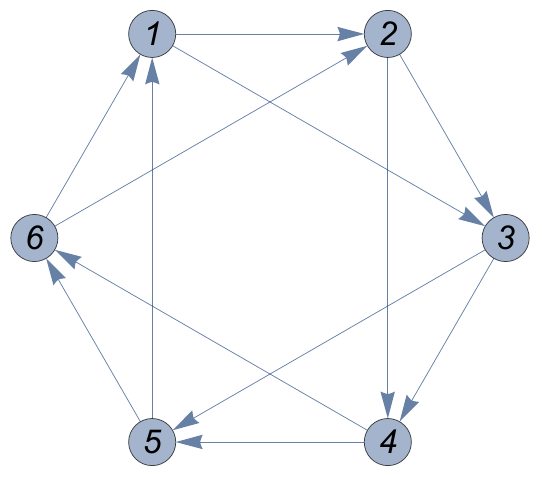}
    \caption{$Q_{dP_3}$}\label{fig:QuiverdP3}
\end{subfigure}\hfil% equal to outside spacing
\begin{subfigure}[t]{0.25\linewidth}
    \centering
    \includegraphics[width=\linewidth]{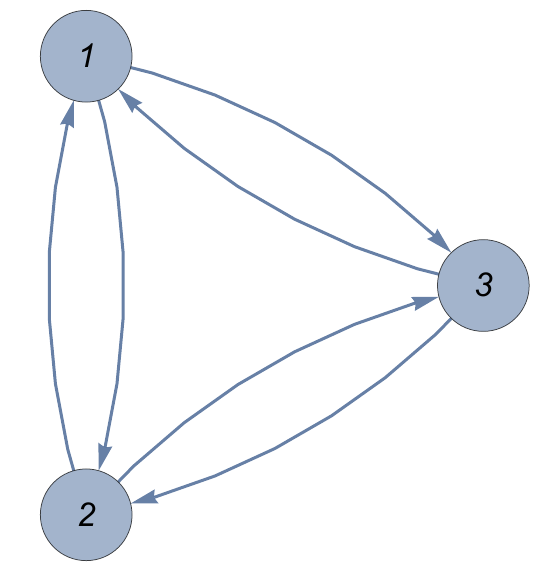}
    \caption{$Q_{dP_3}/\mathbb{Z}_2$}\label{fig:QuiverdP3Z2}
\end{subfigure}\hfil% equal to outside spacing
\begin{subfigure}[t]{0.25\linewidth}
    \centering
    \includegraphics[width=\linewidth]{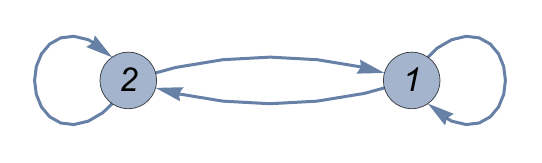}
    \caption{$Q_{dP_3}/\mathbb{Z}_3$}\label{fig:QuiverdP3Z3}
\end{subfigure}\hfil% equal to outside spacing
\caption{Quiver for $\omega_{dP_3}$ and its quotients.}\label{fig:QuiversdP3}
\end{figure}

The brane tiling has a manifest $\mathbb{Z}_6$ symmetry, and consequently the quiver and the potential are invariant under $\mathbb{Z}_6$ cyclic group generated by the permutation $\pi_{dP_3}=(1,2,3,4,5,6)$. We consider the following two subgroups of this $\mathbb{Z}_6$:
\subsubsection*{Case 1: $\mathbb{Z}_2$ generated by the permutation $\pi_{dP_3}^3$}
 The quotient quiver is in Figure \ref{fig:QuiversdP3}(\subref{fig:QuiverdP3Z2}), and the potential is
\begin{equation}
    W_{dP_3/\mathbb{Z}_2}=(X_{12}X_{23}X_{31})^2-(X_{23}X_{32})^2-(X_{13}X_{31})^2-(X_{12}X_{21})^2+X_{13}X_{32}X_{21}+X_{21}X_{13}X_{32} .
\end{equation}
The stable objects  for the corresponding invariant stability conditions on $\D_c(\omega_{dP_3}).$

\subsubsection*{Case 2: $\mathbb{Z}_3$ generated by the permutation $\pi_{dP_3}^2$}
The quotient quiver is in Figure~\ref{fig:QuiversdP3}(\subref{fig:QuiverdP3Z3}), and the potential is
\begin{equation}
    W_{dP_3/\mathbb{Z}_3}=(X_{12}X_{21})^3-X_{21}X_{11}X_{12}X_{22}-X_{11}X_{12}X_{22}X_{21}-X_{12}X_{22}X_{21}X_{11}+(X_{11})^3+(X_{22})^3.
\end{equation}

While we will not do it here, the semistable objects and DT invariants for the corresponding invariant stability conditions on $\omega_{dP_3}$ can in principle be computed from those of the simpler quivers \ref{fig:QuiversdP3}(\subref{fig:QuiverdP3Z2}) and \ref{fig:QuiversdP3}(\subref{fig:QuiverdP3Z3}). These should be compared with the lists of semistables for such invariant stability conditions from \cite{DelMonte2024} and \cite{DML2023}. Note that while $\mathbb{Z}_6$ acts by symmetries of the brane tiling, it does not act by translations. Consequently, the quotienting procedure will not produce a new dimer model, and the quotient quivers in Figure \ref{fig:QuiversdP3}(\subref{fig:QuiverdP3Z3}) are not associated to toric Calabi-Yau threefolds.

\subsection{$Y^{N,0}$ geometries}
\label{banana}
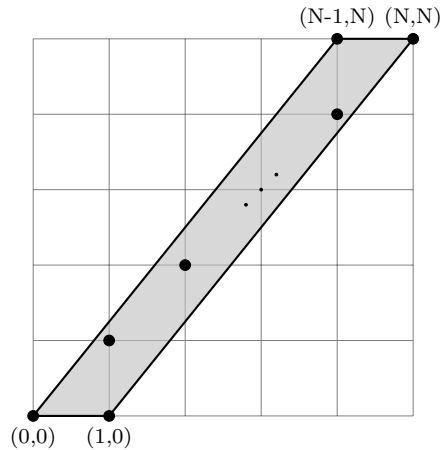
\begin{figure}[b]
\begin{tikzpicture}
    % Grid style
    \draw[very thin, gray] (0,0) grid (5,5);
    
    % Square boundary
    \draw[fill=lightgray, fill opacity=0.6, draw=black,thick] (0,0) -- (1,0) -- (5,5) -- (4,5) -- cycle;
    
    % Vertices of the square
    \filldraw[black] (0,0) circle (2pt) node[anchor=north] {\tiny(0,0)};
    \filldraw[black] (1,0) circle (2pt) node[anchor=north] {\tiny(1,0)};
    \filldraw[black] (5,5) circle (2pt) node[anchor=south] {\tiny(N,N)};
    \filldraw[black] (4,5) circle (2pt) node[anchor=south] {\tiny(N-1,N)};
    \filldraw[black] (1,1) circle (2pt);
    \filldraw[black] (2,2) circle (2pt);
    \filldraw[black] (4,4) circle (2pt);
    \filldraw[black] (2.8,2.8) circle (0.5pt);
        \filldraw[black] (3,3) circle (0.5pt);
            \filldraw[black] (3.2,3.2) circle (0.5pt);
\end{tikzpicture}
\caption{Toric diagram of the $Y^{N,0}$ singularity.}\label{Fig:YN}
\end{figure}
We will conclude our list of examples with the computation of the semistable objects and their DT invariants in a more complicated class of geometries, with $N-1$ compact divisors, $N\in\mathbb{Z}$ $Y^{N,0}$ geometries \cite{Gauntlett2004}. These are  
\begin{figure}[h]
\begin{center}
\begin{subfigure}{.5\textwidth}
\centering
\includegraphics[width=.6\textwidth]{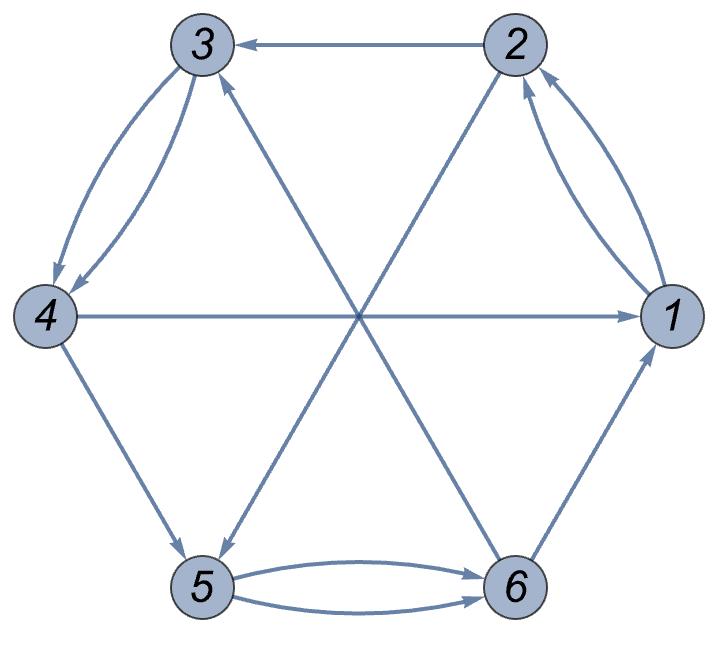}
\caption{Quiver for $Y^{3,0}$}
\label{Fig:QuiverP1P1}
\end{subfigure}\hfill
\begin{subfigure}{.5\textwidth}
\centering
\includegraphics[width=.6\textwidth]{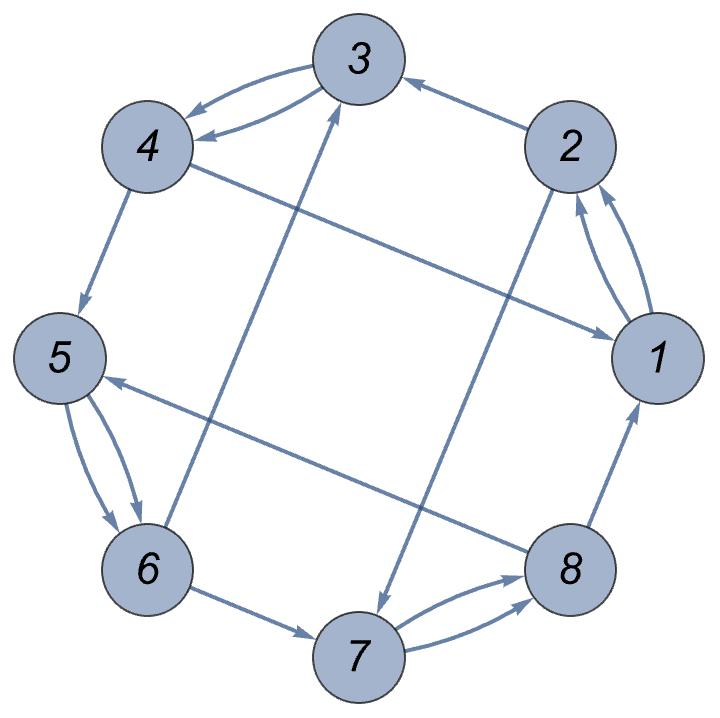}
\caption{Quiver for $Y^{4,0}$}
\label{Fig:SUNQuiver}
\end{subfigure}
\end{center}
\caption{Examples of quivers of $Y^{N,0}$ geometries}\label{fig:YN0}
\end{figure}
an infinite family of toric $CY_3$, 
resolving the so-called $Y^{N,0}$ singularities with toric diagram shown in Figure~\ref{Fig:YN}. As evidenced by their fan, they are resolutions of $\mathbb{Z}_N$ orbifolds of the conifold, so that
the case $N=2$ is the $\omega_{\mathbb{P}^1\times\mathbb{P}^1}$ geometry, while the case $N=1$ is the resolved conifold.

We can construct vectors $\gamma_i$ of Chern characters from those of the conifold by tensoring with characters of $\mathbb{Z}_N$, just as we did for $\omega_{\mathbb{P}^1\times\mathbb{P}^1}$ in Section \ref{sec:StablesF0}. Denote the characters of $\mathbb{Z}_N$ by $\chi_0,\dots,\chi_{N-1}$.

If $\Psi$ is the inverse of the equivalence\footnote{We notice that even if the quotient of the conifold can have more than one resolution for $N$ odd, all projective resolutions are toric by \cite[Theorem~15.1.10]{CLS}. In particular the McKay correspondence from Theorem~\ref{mckay} applies to the projective resolution given by Hilb$^{\mathbb Z_N}(Y)$.} in Theorem~\ref{mckay},
the curve $C\subset \mathcal{O}_{\mathbb{P}^1}^{\oplus 2}$ is the zero section and $y$ is a point of the resolved conifold, then
\begin{equation}
\gamma_{2k+1}=\ch\left(\Psi\left(\mathcal{O}_C\otimes\chi_k \right) \right),\qquad \gamma_{2k}=\ch\left(\Psi\left(\mathcal{O}_C(-1)[1]\otimes\chi_k \right) \right),
\end{equation}
\begin{equation}
\gamma_{k}+\gamma_{k+1}=\ch\left(\Psi\left(\mathcal{O}_y\otimes\chi_k \right) \right),\qquad k=1,\dots,N.
\end{equation}
We define combinations
\begin{equation}
    \delta:=\sum_{k=1}^{2N}\gamma_i=\ch\left(\mathcal{O}_p \right),\qquad v_j:=\gamma_{2j-1}+\gamma_{2j}, \quad 1\leq j\leq N
\end{equation}

Essentially the same argument as the proof of Theorem \ref{thm:DTP1P1} leads to the following.

\begin{thm}The nonzero DT invariants for the $\mathbb{Z}_N$-invariant stability conditions on the $Y^{N,0}$ Calabi-Yau threefolds are
    \[\Omega(\gamma_{2j-1}+nv_j)= \Omega(\gamma_{2j}+n v_j)=1, \qquad n\in \bZ,\qquad j=1,\dots,N,\]
    \[\Omega\left(\pm\sum_{j=a}^bv_j+n\delta\right)=-2, \qquad n\in \bZ,\qquad 0<a\le b<N-1,\]\[ \Omega( n\delta)=-2N, \qquad n\in \bZ\setminus\{0\}.\]
    \end{thm}

    \begin{proof}
        As in the case of local $\bP^1\times \bP^1$, we can immediately see that on all rays except the central one lie exactly $N$ stable spherical objects of the form $\Psi\left(\mathcal{O}_C(n)\otimes\chi_p \right)$, $p=0,\dots,N-1$, with DT invariants $\Omega=1$. The semistable objects on the central ray will again be zero-dimensional sheaves on the orbifold. The degree zero DT generating function is \cite{BCY}
\begin{equation}
    \underline{\DT}_0(\cY)=\prod_{m\geq 1} \left[(1-q_0^m q_1^m)^{-2Nm}\cdot\prod_{0<a\le b<N}\left(1-q^m\prod_{j=a}^bq_j\right)^{-2m}\cdot\left(1-q^m\prod_{j=a}^bq_j^{-1}\right)^{-2m} \right],
\end{equation}
from which we can read off $\Omega(n\delta)=-2N$ and $\Omega\left(\pm\sum_{j=a}^bv_j+n\delta\right)=-2$, completing the proof.
    \end{proof}
    
For $N=1$, this formula reproduces the DT invariants for the resolved conifold, for $N=2$ it reproduces the DT invariants for the $\mathbb{Z}_2$-invariant stability conditions on $\omega_{\mathbb{P}^1\times\mathbb{P}^1}$.
As remarked in the introduction, the invariant stability conditions for $Y^{N,0}$ geometries give an infinite class of examples of analytic wall-crossing structures associated with local Calabi-Yau threefolds with  compact divisors, and also an infinite class of examples of trivial solutions to DT Riemann-Hilbert problems. 

We could also have worked directly in terms of quivers with potential, following the algebraic approach of Section  \ref{sec:AlgApproach}. For our present case, these have been obtained by brane tiling techniques \cite{Benvenuti2004,Franco2006}. In particular, we consider the ``uniform dimer model'' constructed in \cite{Bershtein2019}, obtained by quotienting the square bipartite tiling of $\mathbb{Z}^2$ by the lattice $\Gamma=\mathbb{Z}\left( \begin{array}{c}
N \\ N
\end{array} \right)\oplus \mathbb{Z}\left( \begin{array}{c}
0 \\ 2
\end{array} \right). $ 
They give rise to quivers with adjacency matrix
\begin{align}
B_{k,k+1}= \begin{cases}
2, &k\text{ odd,} \\
1, & k\text{ even},
\end{cases}
&&
B_{k,k+3}= \begin{cases}
-1, & k\text{ odd} \\
0 & k\text{ even},
\end{cases}\nonumber
\end{align}
\begin{align}\label{eq:BMatrixYN}
B_{1,2N}=-1,\quad B_{3,2N}=-1,\quad B_{2N-1,2N}=2,
\end{align}
and they are shown in Figure~\ref{fig:YN0}(\subref{Fig:SUNQuiver}) for the cases $N=3,4$. The potential is \cite{Benvenuti2004}
{\small\begin{equation}
    W_{Y^{N,0}}=\sum_{k=1}^N\left(X_{2k-1,2k}^{(1)}X_{2k,2k+1}X_{2k+1,2k+2}^{(2)}X_{2k+2,2k-1}-X_{2k-1,2k}^{(2)}X_{2k,2k+1}X_{2k+1,2k+2}^{(1)}X_{2k+2,2k-1} \right),
\end{equation}}
where the indices of the potential terms should be considered modulo $2N$. The $\mathbb{Z}_N$ group acts on the quiver labels as the permutation $\pi_{Y^{N,0}}=(1,3,\dots,N)(2,4,\dots,2N)$, that also leaves $W_{Y^{N,0}}$ invariant, and it is easy to check that the potential obtained by the quotient is
\begin{equation}
    \frac{1}{N}W_{Y^{N,0}/\mathbb{Z}_N}=W_{\rm conifold}.
\end{equation}

\bibliographystyle{alph}
\bibliography{Biblio}

\end{document}